\documentclass[]{siamart0516}

\usepackage{amsmath,amsfonts,amssymb}
\usepackage{draftwatermark}
\usepackage{stmaryrd}
\usepackage{boxedminipage}
\usepackage{algorithmic}
\usepackage{url}
\usepackage{subfigure}
\usepackage{tikz}

\usepackage[mathscr]{eucal}

\usepackage{amsbsy}

\usepackage{bm}

\usepackage{paralist}

\usepackage{xspace}

\newtheorem{example}[theorem]{Example}

\newcommand{\twist}[1]{{\tt twist}(#1)}

\newcommand{\sq}[1]{{\tt sq}(#1)}

\newcommand{\bfrho}{{\boldsymbol \rho}}
\newcommand{\bfdel}{{\boldsymbol \delta}}

\newcommand{\V}[1]{{\bm{\mathbf{\MakeLowercase{#1}}}}} 

\newcommand{\M}[1]{{\bm{\mathbf{\MakeUppercase{#1}}}}} 

\newcommand{\T}[1]{\boldsymbol{\mathscr{\MakeUppercase{#1}}}} 

\newcommand{\TA}{\T{A}}
\newcommand{\TB}{\T{B}}
\newcommand{\TC}{\T{C}}
\newcommand{\TX}{\T{X}}
\newcommand{\TU}{\T{U}}
\newcommand{\TV}{\T{V}}
\newcommand{\TS}{\T{S}}
\newcommand{\TW}{\T{W}}
\newcommand{\TY}{\T{Y}}
\newcommand{\TQ}{\T{Q}}

\newcommand{\TD}{\T{D}}

\newcommand{\TG}{\T{G}}
\newcommand{\TI}{\T{I}}
\newcommand{\TP}{\T{P}}

\newcommand{\Vu}{\V{u}}
\newcommand{\Vv}{\V{v}}

\newcommand{\MG}{\M{G}}
\newcommand{\MZ}{\M{Z}}
\newcommand{\MM}{\M{M}}
\newcommand{\MA}{\M{A}}
\newcommand{\MX}{\M{X}}
\newcommand{\MS}{\M{S}}
\newcommand{\MV}{\M{V}}
\newcommand{\MU}{\M{U}}
\newcommand{\MD}{\M{D}}
\newcommand{\MB}{\M{B}}

\newcommand{\MQ}{\M{Q}}
\newcommand{\MW}{\M{W}}

\newcommand{\bfa}{{\bf a}}
\newcommand{\MR}{\M{R}}

\newcommand{\tp}{^{\rm P}} 

\newcommand{\starM}{{\star_{\rm M}}}
\newcommand{\starB}{{\star_{\rm B}}}

 \newcommand{\bea}{ \left[ \begin{matrix} }
 \newcommand{\eea}{ \end{matrix} \right] }

\SetWatermarkText{ }
\SetWatermarkScale{3}

\definecolor{blue}{rgb}{0,0,1}
\definecolor{red}{rgb}{1,0,0}
\definecolor{green}{rgb}{.5,.8,.5}

\newcommand\LH[1]{\textcolor{green}{LH: #1}}

\title{Tensor-Tensor Products for Optimal Representation and Compression} 
\author{Misha Kilmer%
\thanks{Department of Mathematics, Tufts University, Medford, MA (\email{misha.kilmer@tufts.edu}).}%
\and
Lior Horesh%
\thanks{Mathematics of AI, IBM Thomas J. Watson Research Center, Yorktown Heights, NY (\email{lhoresh@us.ibm.com}).}%
\and
Haim Avron%
\thanks{Department of Applied Mathematics, Tel Aviv university, Tel Aviv-Yafo, Israel, (\email{haimav@tauex.tau.ac.il}).}%
\and
Elizabeth Newman%
\thanks{Department of Mathematics, Emory University, Atlanta, GA (\email{elizabeth.newman@emory.edu}).}%
}
\headers{Tensor Optimality}
{M. Kilmer, L. Horesh, H. Avron, E. Newman}

\begin{document}
\maketitle


\begin{abstract}
In this era of big data, data analytics and machine learning, it is imperative to find ways to compress large data sets such that intrinsic features necessary for subsequent analysis are not lost.  The traditional workhorse for data dimensionality reduction and feature extraction has been the matrix SVD, which presupposes that the data has been arranged in matrix format.  Our main goal in this study is to show that high-dimensional data sets are more compressible when treated as tensors (aka multiway arrays) and compressed via tensor-SVDs under the 
tensor-tensor product structures in \cite{KilmerMartin2011,LAA}. 
We begin by proving 
Eckart Young optimality results for families of 
tensor-SVDs under two different truncation strategies.  
As such optimality properties can be proven in both matrix and tensor-based algebras, a fundamental question arises: does the tensor construct subsume the matrix construct in terms of representation efficiency?  The answer is yes, as shown when we prove that a tensor-tensor representation of an equal dimensional spanning space can be superior to its matrix counterpart.   We then investigate how the compressed
representation provided by the truncated tensor-SVD is related both theoretically and in compression performance to its closest tensor-based analogue, truncated HOSVD \cite{DeLath2000,DeLathDimRedux04},
thereby showing the potential advantages of our tensor-based algorithms.
Finally, we propose new tensor truncated SVD variants, namely multi-way tensor SVDs, provide further approximated representation efficiency
and discuss under which conditions they are considered optimal.  We conclude with a numerical study demonstrating the utility of the theory.

\end{abstract}

\section{Significance}
 Much of real-world data is inherently multidimensional, often involving high-dimensional correlations. 
 However, many data analysis pipelines process data as two-dimensional arrays (i.e., matrices) even if the data is naturally represented in high-dimensional.
 The common practice of matricizing high-dimensional data is due to the ubiquitousness and strong theoretical foundations of matrix algebra.  
 
 Over the last century, dating back to 1927~\cite{hitchcock1927tensor} with the introduction of the canonical (CP) decomposition, various tensor-based approximation 
 techniques have been developed. 
 These high-dimensional techniques have demonstrated to be instrumental in a broad range of application area, yet, hitherto, none have been theoretically proven to outperform matricization in general settings. This lack of matrix-mimetic properties and theoretical guarantees has been impeding adaptation of tensor-based techniques as viable mainstream data analysis alternatives. 
 
 
 %
 In this study, we propose preserving data in a native, tensor-based format while processing it using new matrix-mimetic, tensor-algebraic formulations.  Considering a general family of tensor algebras, we prove an Eckart-Young optimality theorem for truncated tensor representations.  
Perhaps more significantly, we prove these tensor-based reductions are superior to traditional matrix-based representations.  
 Such results distinguish the proposed approach from 
 other tensor-based approaches.  
We believe this work will lead to revolutionary new ways in which data with high-dimensional correlations are treated.

%

\section{Introduction}

\subsection{Overview}
Following the discovery of the spectral decomposition by 
Lagrange in 1762, the Singular Value Decomposition (SVD) was discovered independently by  
Beltrami and 
Jordan, in 1873 and 1874 respectively. Further generalization of the decomposition came independently by  
Sylvester in 1889, and by Autonne in 1915 \cite{stewart1993early}. Yet, probably one of the most notable theoretical results associated with he decomposition, is due to Eckart and 
Young that provided the first optimality proof of the decomposition back in 1936 \cite{eckart1936approximation}.


The use of SVD 
in data analysis is ubiquitous. From a statistical point of view, the singular vectors of a (mean subtracted) data matrix represent the principle component directions; the directions in which there is maximum variance correspond to the largest singular values. 
However, the SVD is historically well-motivated by spectral analysis of linear operators.  
Since data matrices that arise in typical data analysis tasks are essentially just rectangular arrays of data which may not correspond directly to either statistical interpretation nor representations of linear transforms, the prevalence of the the SVD in data analysis applications requires further investigation.

The utility of the SVD in the context of data analysis is due to two key factors: the well-known Eckart-Young theorem (also known as the Eckart-Young-Minsky theorem), and the fact that the SVD (or in some cases a partial decomposition or high-fidelity approximation) can be efficiently 
computed (see, for example, \cite{Watkins} and references therein), relative to the matrix dimensions and/or desired rank of the partial decomposition.  Formally, the Eckart-Young theorem gives the solution to the problem of finding the best (in the Frobenius norm or 2-norm) of a rank $k$ approximation to a matrix with rank greater than $k$, in terms of the first $k$ terms of the SVD expansion of that matrix. 
The theorem implies, in some informal sense, that the majority of the informational content is contained by the dominant singular subspaces (i.e. the span of the singular vectors corresponding to the largest singular values), opening the door for compression, efficient representation, de-noising, etc. The SVD motivates modeling data as matrices, even in cases where the more natural model is an high dimensional array (a tensor) -- a process known as matricization.   

Nevertheless, there is intuitively an inherent disadvantage of matricization of data which could naturally be represented as a tensor. For example, a gray scale image is naturally represented as a matrix of numbers, but a video is naturally represented as tensor since there is an additional dimension - time.  Preservation of the dimensional integrity of the data can be imperative for subsequent processing of the data while accounting for high-dimensional correlations embedded in the structures the data is organized in. Even so, in practice, there is often a surprising dichotomy between the data representation and the algebraic constructs employed for its  processing. Thus, in the last century there has been efforts to define decomposition of tensorial structures, e.g. CP~\cite{hitchcock1927tensor,CarrollChang70,Harshman70}, Tucker~\cite{Tuck1963a}, HOSVD \cite{DeLath2000} and Tensor-Train~\cite{oseledets2011tensor}. However, for none of the aforementioned decompositions there is known Eckart-Young-like result. 

In this study, we attempt to close this gap by proving a general Eckart-Young optimality theorem for a tensor-truncated representation. We also consider the  optimal data compression with a third-order (or higher order) tensor vs optimal compression of the same data oriented in matrix form.  

An Eckart-Young-like theorem must revolve around some tensor decomposition and a metric.  We consider tensor decompostions built around the idea of the t-product presented in \cite{KilmerMartin2011} and the tensor-tensor product extensions of a similar vein in \cite{LAA}.  In their original work, the authors define a tensor-tensor product between third-order tensors, and corresponding algebra in which notions of identity, orthogonality, and transpose are all well-defined.  These ideas were generalized to higher order tensors in \cite{Martinetal2013}.  These lead to the definition of a tensor-SVD based on the t-product, for which the authors show there exists an Eckart-Young type of optimality result in the Frobenius norm.  Other popular tensor decompositions (e.g. Tucker, HOSVD, CP,  \cite{Tuck1963a,DeLath2000,hitchcock1927tensor,KoldaBader}) do not lend themselves easily to a similar analysis, either because direct truncation does not lead to an optimal lower-term approximation, or because they lack analogous definition of orthogonality and energy preservation.  In this work, we show that the HOSVD can be interpreted as a special case of the class of tensor-tensor products we consider, and as such, we are able to explain why truncation of the HOSVD will, in general, not give an optimal compressed representation when compared to the proposed truncated tensor SVD approach.

The t-product is orientation dependent, meaning that the ability to compress the data in a meaningful way depends on the orientation of the tensor (e.g. how the data is organized in the tensor).
Consider, for example, a collection of $\ell$ gray scale images of size $m \times n$.   Those images can be placed into a tensor as lateral slices (resulting in an $m \times \ell \times n$ array), but they can also be rotated first and then placed as lateral slices (resulting in a $n \times \ell \times m$ array),  etc.   In many applications, there are good reasons to keep the second, lateral, dimension fixed (i.e. representing time, total number of samples, etc.), but in others there may be no obvious reason to preferentially treat one of the other two dimensions.   Thus, in this paper we also consider variants of the original t-product approach that can give optimal approximations to the tensorized data without handling
one spatial orientation
differently than another, 
while still offering improved representation over 
treating the data in matricized format.  

We primarily limit the discussion to third-order tensors, though in the final section we discuss how the ideas generalize to higher-order as well.  Indeed, the potential
for even greater compression gain for higher-order representations of the data exists, provided the data has higher-dimensional correlations to be exploited.

 

\subsection{Paper Organization}
In \Cref{sec:back}, we give background notation and definitions.  In \Cref{sub:t-svd} we define the t-SVDM, and prove an Eckart-Young-like theorem for it.  
\Cref{sec:latent} and \Cref{sec:superior} are devoted to a theorem and discussion about when and why some data is more amenable to optimal representation through use of the truncated t-SVDM than through the matrix SVD.  In \Cref{sec:new}, we propose three new approaches to further compress the third order data by a more equal treatment of the first and third dimensions.  We 
show in  \Cref{sec:hosvd} that the HOSVD decomposition can be interpreted as a specific case of the proposed 
tensor-tensor product framework. Using this observation, we can prove that the truncated HOSVD cannot in general provide as much compression as our approach.  \Cref{sec:new} discusses multisided tensor compression. \Cref{sec:numerical} contains a numerical study and highlights extentions to higher order data.  A summary and future work are the subject of \Cref{sec:conclusions}.

\section{Background}  \label{sec:back}

For the purposes of this paper, a tensor is a multi-dimensional array, and the order of the tensor is defined as the number of dimensions of this array.   As we are concerned with third-order tensors throughout most of the discussion, we limit our notation and definitions to the 
third-order case here, and generalize to higher order in the final section.

\subsection{Notation and Indexing}
 A third-order tensor $\TA$ is an object  in $\mathbb{C}^{m \times p \times n}$.  
Its Frobenius norm, $\| \TA \|_F$, is analogous to the matrix case:  that is $\| \TA \|_F^2 = \sum_{i,j,k} |\TA_{i,j,k}|^2$. 
We use {\sc Matlab} notation for entries: $\TA_{i,j,k}$ denotes the entry at row $i$ and column $j$ of the matrix going $k$ ``inward".   
The fibers of tensor $\TA$ are defined by fixing two indices. Of note are the tube fibers,  written as
 $\TA_{i,j,:}$ or $\V{a}_{i,j}$, $i=1\!\!:\!\!m, j=1\!\!:\!\!p$. A slice of a third-order tensor $\TA$ is a two-dimensional array defined by fixing one index. 
Of particular note are the frontal and lateral slices, as depicted in \Cref{fig:ten}.  The $i^{th}$ frontal slice is expressed as 
$\TA_{:,:,i}$ and also referenced as $\MA^{(i)}$ for convenience in later definitions.  The $j^{th}$ lateral slice would be $\TA_{:,j,:}$ or equivalently expressed as $\vec{\TA}_j$.  

\begin{figure*}[t]\label{fig:ten}
\begin{center}
\includegraphics[scale=.4]{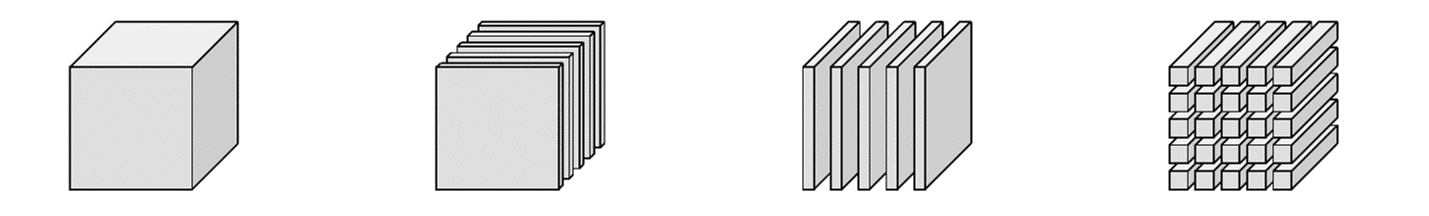}
\end{center}
\caption{Fibers and slices of a $m \times p \times n$ tensor $\TA$.  Left to right:  Frontal slices, referenced as either $\MA^{(i)}$ or $\TA_{:,:,i}$; lateral slices denoted as either $\vec{\TA}_j$ or $\TA_{:,j,:}$; tube fibers $\V{a}_{ij}$ or $\TA_{i,j,:}$ .} 
\end{figure*}

Some other notation that we use for convenience are the {\tt vec} and {\tt reshape} operators that map matrices to vectors by column unwrapping, and vice versa:
\[ {\bfa}  = \mbox{\tt vec}(\MA) \in \mathbb{C}^{mn}  \leftrightarrow \MA = \mbox{\tt reshape}({\bfa},[m,n]) .\]
We can also define invertible mappings between $m \times n$ matrices and $m \times 1 \times n$ tensors 
by twisting and squeezing\footnote{If $\TA$ is $m \times 1 \times n$, the {\sc Matlab} command $\mbox{\tt squeeze}(\TA)$ returns the $m \times n$ matrix.} (\cite{KilmerEtAl2013}): i.e. $\MX \in \mathbb{C}^{m \times n}$ is related to $\vec{\TX}$ via
\[  \vec{\TX} = \mbox{\tt twist}(\MX)  \mbox{ and }  \MX = \sq{\vec{\TX}}. \]

The mode-1, mode-2, and mode-3 unfoldings of $\TA \in \mathbb{C}^{m \times p \times n}$ are $m \times np, p \times mn$, and $n \times pn$ respectively, and are given by
\begin{equation} \begin{array}{lcl} \TA_{(1)} & := & [\MA^{(1)},\ldots, \MA^{(n)}] \\
                         \TA_{(2)} & := & [(\MA^{(1)})^\top,\ldots, (\MA^{(n)})^\top] \\ 
 \TA_{(3)} & := & [\sq{ \TA_{:,1,:}}^\top, \sq{ \TA_{:,2,:}}^\top, \ldots , \sq{\TA_{:,p,:}}^\top ]  \end{array} \end{equation}
These are useful in defining mode-wise tensor-matrix products (see \cite{KoldaBader}).  For example, $\TA \times_1 \M{F}$ for $\M{F} \in \mathbb{C}^{k \times m}$ is
equivalent to computing $\M{F} \TA_{(1)}$ and reshaping the result to a $k \times p \times n$ tensor.   
 $\TA \times_3 \MM$ for a $r \times n$  means we compute matrix-matrix product $\M{M} \TA_{(3)}$, 
which is $r \times mp$, and then reshape the result to an $m  \times p \times r$ tensor.  

The HOSVD \cite{DeLath2000} can be expressed as 
\begin{equation} \label{eq:hosvdfull} \TA = \TC \times_1 \MU \times_2 \MV \times_3 \MW, \end{equation}
where $\MU, \MV, \MW$ are the matrices containing the left singular vectors, corresponding to non-zero singular values, of the matrix SVDs of $\TA_{(1)}, \TA_{(2)}, \TA_{(3)}$, respectively. For an $m \times p \times n$ tensor, $\MU$ would be $m \times r_1$, $\MV$, $p \times r_2$, and $\MW$ $n \times r_3$, where $r_1, r_2, r_3$ are the ranks of the three respective unfoldings.  Correspondingly, we say the tensor has HOSVD rank $(r_1,r_2,r_3)$.  The $r_1 \times r_2 \times r_3$ core tensor is given by $\TC := \TA \times_1 \MU^* \times_2 \MV^* \times_3 \MW^*$.   While the columns of the factor matrices are orthonormal, the core need not be diagonal, and its entries need not be non-negative.  In practice, compression is done by truncating to an HOSVD rank $(k_1, k_2, k_3)$, but unlike the matrix case, such truncation does not lead to an optimal truncated approximation in a norm sense.   

If $\V{a}, \V{b}, \V{c}$ are length $m, n, p$ vectors, respectively, then $ \TB := \V{u} \circ \V{v} \circ \V{w}$ is called a {\bf rank-1} tensor, and $\TB_{i,j,k} = u_i v_j w_k$.  A CANDECOMP/PARAFAC (CP) \cite{hitchcock1927tensor,CarrollChang70,Harshman70} decomposition of a tensor $\TA$ is an expression as a sum of rank-1 outer-products:
\[ \TA  = \sum_{i=1}^r \V{u}^{(i)} \circ \V{v}^{(i)} \circ \V{w}^{(i)} := \llbracket \MU, \MV, \MW \rrbracket,\]
where the factor matrices $\MU, \MV, \MW$ have the $\V{u}^{(i)}, \V{v}^{(i)}, \V{w}^{(i)}$ as their columns, respectively.
{\it If $r$ is minimal, then $r$ is said to be the rank of the tensor; that is $r$ is the tensor rank}. Note that although it is known that $r \le \min(np,mp,nm)$, 
determining the rank of a tensor is an NP-hard problem \cite{HillarLim2013}.  Moreover, the factor matrices need not have orthonormal columns, nor even be full rank.   

However, if we are given a set of factor matrices with $k$ columns such that $\TA = \llbracket \MU, \MV, \MW \rrbracket$ holds, this {\it is still a CP decomposition, even if we do not know if $k$ corresponds to, or is bigger than, the rank}. In other words, given a CP decomposition where the factor matrices have $k$ columns, all we know is that the rank of the tensor represented by this decomposition has tensor rank at most $k$.  

In the remaining subsections of this section, we give background on the tensor-tensor product representation recently developed in the literature \cite{KilmerMartin2011,KilmerEtAl2013,LAA}.   The goal of this paper is to derive provably optimal (i.e. minimal, in the Frobenius norm) approximations to tensors under this framework.  We will compare these approximations to those derived by processing the same data in matrix form, and show links between our approximations and HOSVD and CP decompositions.

\subsection{A Family of Tensor-Tensor Products}
 The first closed multiplicative operation between a pair of third-order tensors of appropriate dimension was given in \cite{KilmerMartin2011}.  That operation was named the {\it t-product}, and the resulting linear algebraic framework is 
 described in \cite{KilmerMartin2011,KilmerEtAl2013}.   
In \cite{LAA}, the authors continued the theme from \cite{KilmerMartin2011} by describing new families of tensor-tensor products, and their associated algebraic framework.   As the presentation in \cite{LAA} includes the t-product as one example, we will introduce the class of tensor-tensor products of interest, and at times throughout the document, highlight the t-product as a special case. 

Let $\MM$ be any invertible, $n\times n$ matrix, and $\TA \in \mathbb{C}^{m \times p \times n}$.  We will use hat notation to denote the tensor in the transform domain specified by $\MM$:
\[ \widehat{\TA} : = \TA \times_3 \MM, \]
where, since $\MM$ is $n \times n$, $\widehat{\TA}$ has the same dimension as $\TA$.  
Importantly, {\it $\widehat{\TA}$ corresponds to applying $\MM$ along all tube fibers}, although it is implemented according to the definition of computing the matrix-matrix product $\MM \TA_{(3)}$, and reshaping the result.
We note that the ``hat" notation should be understood in context of the linear transformation applied, meaning that 
the transform-domain version of $\TA$ depends on the transform, $\MM$. 


\begin{algorithm} 
\caption{\label{alg:prod} Algorithm $\TA \starM \TB$ for invertible $\MM$ from \cite{LAA}}
 \begin{algorithmic}[1]
   \STATE Define $\widehat{\TA} := \TA \times_3 \MM$, $\widehat{\TB} := \TA \times_3 \MM$, \\
   \FORALL{$i=1,\ldots,n$}
  \STATE $\widehat{\TC}_{:,:,i} = \widehat{\TA}_{:,:,i} \widehat{\TB}_{:,:,i}$ \\
\ENDFOR
   \STATE Define $\TC = \widehat{\TC} \times_3 \MM^{-1}$.  Now $\TC = \TA \starM \TB$. 
\end{algorithmic}
\end{algorithm}

From \cite{LAA}, we define the $\starM$ product between $\TA \in \mathbb{C}^{m \times p \times n}$ and $\TB \in \mathbb{C}^{p \times r \times n}$ through the steps in \Cref{alg:prod}. 
Step 2 is 
embarrassingly parallelizable since the matrix-matrix products in the loop are independent.  Step 1 (and 3) could in theory also be performed in $p$ independent blocks of matrix-matrix products (or matrix solves, in the case of $\times_3 \MM^{-1}$ to avoid inverse computation).

Choosing $\MM$ as the (unnormalized) DFT matrix and comparing to the algorithm in the previous section, we see that $\TA \starM \TB$ effectively reduce to 
the t-product operation, $\TA * \TB$, defined in \cite{KilmerMartin2011}.    Thus, {\it the t-product is a special instance of products from the $\starM$ product family.}

\subsection{Tensor Algebraic Framework}
Now that we have many possible options for computing a tensor-tensor product, we can introduce the remaining parts of the framework.
For the t-product, the linear algebraic framework is in \cite{KilmerMartin2011,KilmerEtAl2013}.    However, as the t-product is a special case of the $\starM$ as noted above, we will elucidate here the linear algebraic framework as described in \cite{LAA}, with pointers to \cite{KilmerMartin2011, KilmerEtAl2013} so we can cover all cases.  \\

There exists an identity element and a notion of conjugate transposition:
\begin{definition}[Identity Tensor; Unit-normalized Slices]
	The $m \times m \times n$ identity tensor $\T{I}$ satisfies  $\TA \starM \TI = \TA = \TI \starM \TA$ for $\TA \in \mathbb{C}^{m \times m \times n}$.  For invertible $\MM$, this tensor always exists. If $\vec{\TB}$ is $m \times 1 \times n$, and
$\vec{\TB}^{\rm H} \starM \vec{\TB}$ is the $1 \times 1 \times n$ identity tensor under $\starM$, we say the $\vec{\TB}$ is a unit-normalized tensor slice.
\end{definition}

\begin{definition}[conjugate transpose] 
	Given $\TA \in \mathbb{C}^{m \times p \times n}$ its $p \times m \times n$ {\bf conjugate transpose} under $\starM$ $\TA^{\rm H}$ is defined such that
\[ (\widehat{\TA}^{\rm H})_{:,:,i} = (\widehat{\TA}_{:,:,i})^{\rm H}, \qquad i = 1,\ldots, n . \]
\end{definition}	
As noted in \cite{LAA}, this definition ensures the multiplication reversal property for the Hermitian transpose under $\starM$:
$ \TA^{\rm H} \starM \TB^{\rm H} = (\TB \starM \TA)^{\rm H}. $   This definition is consistent with the t-product transpose given in \cite{KilmerMartin2011} when $\starM$ is defined by the DFT matrix.

With conjugate transpose and an identity operators defined, the concept of unitary and orthogonal tensors is now straightforward:
\begin{definition}[unitary/orthogonality]
	Two, $m \times 1 \times n$ tensors, $\vec{\TA}, \vec{\TB}$ are called $\starM$-orthogonal slices if $\vec{\TA}^{\rm H} \starM \vec{\TB}$ is the tube fiber $\bf 0$.  If $\T{Q} \in \mathbb{C}^{m \times m \times n}$ ($\TQ \in \mathbb{R}^{m \times m \times n}$) is called $\starM$-unitary\footnote{The reader should regard of the elements of the tensor as $1 \times 1 \times n$ tube fibers under $\starM$.  This forms a free module (see \cite{LAA}).  The analogy to elemental inner-product like definitions over the underling free module induced by $\starM$ and space of tube-fibers is referenced in Section 4 of \cite{LAA} for general $\starM$ products, and in Section 3 of \cite{KilmerEtAl2013} for the t-product.   The notion of orthogonal/unitary tensors
is therefore consistent with this generalization of inner-products, which is captured in the first part of the definition.},  ($\starM$-orthogonal) if
	\begin{equation*}
		\T{Q}^{\rm H} \starM \T{Q} = \T{I} = \T{Q} \starM \T{Q}^{\rm H},
	\end{equation*}
	where $H$ is replaced by transpose for real-valued tensors.  Note that $\T{I}$ must be the one defined under $\starM$ as well, and that tube fibers on the diagonal correspond to unit-normalized slices $\vec{\TQ}_j^{\rm H} \starM \vec{\TQ}_j$ and off-diagonals
to tube fibers formed from products $\vec{\TQ}_j^{\rm H} \starM \vec{\TQ}_k$ with $k \not= j$.      
\end{definition}


\section{Tensor $\starM$ SVDs and Optimal Truncated Representation}\label{sub:t-svd}

As noted in the previous section and described in more detail in \cite{LAA}, any invertible matrix $\MM$ can be used to define a valid tensor-tensor product.  However, in this paper, we will focus on a specific class of matrices $\MM$ for which
unitary-invariance under the Frobenius norm is preserved, as we  discuss in the following.  We then use this feature to develop Eckart-Young theory for our tensor decompositions later in this section.

\subsection{Unitary Invariance}
Unitary invariance of real-valued orthogonal tensors under the t-product was shown in \cite{KilmerMartin2011}.  Here, we prove a
more general result. 
\begin{theorem} \label{th:invariance} With the choice of $n \times n$ $\MM = c \M{W}$ for unitary (orthogonal) $\M{W}$, and non-zero $c$,  assume $\TQ$ is $m \times m \times n$ and $\starM$-unitary ($\starM$-orthogonal).   Then 
$$\| \TQ \starM \TB \|_F = \| \TB \|_F,  \TB \in \mathbb{C}^{m \times k \times n}. $$  Likewise, if 
$\TB \in \mathbb{C}^{p \times m \times n}$, $\| \TB \starM \TQ \|_F = \| \TB \|_F$. \end{theorem}
\begin{proof} 

Suppose $\MM = c \M{W}$ where $\M{W}$ is unitary.  Then $\MM^{-1} = \frac{1}{c} \M{W}^{\rm H}$.  
Next, 
\begin{equation} \label{eq:norm} \| \widehat{\TB} \|_F = \| \TB \starM \MM \|_F = \| c \M{W} \TB_{(3)} \|_F = c \| \TB_{(3)} \|_F = c \| \TB \|_F. \end{equation}

Let $\TC = \TQ \starM \TB$.   Using (\ref{eq:norm}),
\[
  \| \TB \|_F^2 = \frac{1}{|c|^2} \| \widehat{\TB} \|_F^2 = \frac{1}{|c|^2} \sum_{i=1}^p \| \widehat{\TQ}_{:,:,i} \widehat{\TB}_{:,:,i} \|_F^2 = \frac{1}{|c|^2} \| \widehat{\TC} \|_F^2 = \| \TC \|_F^2 = \| \TQ \starM \TB \|_F^2 ,\]
since each $\widehat{\TQ}_{:,:,i}$ must be unitary.   The other direction follows similarly.   
\end{proof}

We now have the framework we need to describe tensor SVDs induced by a fixed, $\star_M$ operator.   These were defined and existence was proven in \cite{KilmerMartin2011} for the t-product over real tensors, and in \cite{LAA} for $\starM$ more generally. 
\begin{definition}[\cite{KilmerMartin2011,LAA}]\label{def:tsvd}
	Let $\TA$ be a $m \times p \times n$ tensor. The (full) $\starM$ tensor SVD (t-SVDM) of $\TA$ is
\begin{equation} \label{eq:TSVDMr}
\TA   = \TU \starM \TS \starM \TV^{\rm H}
          =  \sum_{i=1}^{r} \TU_{:,i,:} \starM \TS_{i,i,:} \starM \TV_{:,i,:}^{\rm H} 
                  \end{equation}
    where $\TU \in \mathbb{R}^{m \times m \times n}$, $\TV \in \mathbb{R}^{p \times p \times n}$ are $\starM-unitary$, and $\TS \in \mathbb{R}^{m \times p \times n}$ is a tensor whose frontal slices  are diagonal (such a tensor is called {\bf f-diagonal}), and $r \le \min(m,p)$ is the number of non-zero tubes in $\TS$.   When $M$ is the DFT matrix, this reduces to the t-product-based t-SVD introduced in \cite{KilmerMartin2011}.
\end{definition}

Clearly, if $m > p$, from the second equality we can get a {\it reduced} t-SVDM, by restricting $\TU$ to have only $p$ orthonormal lateral slices, and $\T{S}$ to be $p \times p \times n$, as opposed to the full representation.  Similarly, if $p > m$, we only need to keep the $m \times m \times n$ portion of $\TS$ and the $m$ columns of $\TV$ to obtain the same representation.  
An illustration of the decomposition is presented in the top part of \Cref{fig:tsvdII}.

\begin{algorithm}[h] 
\caption{ \label{alg:one} Full t-SVDM from \cite{LAA}.}
\begin{algorithmic}[1]
 \STATE  $\widehat{\TA} \leftarrow \TA \times_M \MM$ \\
 \FORALL{$i=1,\ldots,n$}  
 \STATE $[\widehat{\TU}_{:,:,i},\widehat{\TS}_{:,:,i}, \widehat{\TV}_{:,:,i}] = \mbox{\tt svd} ([\widehat{\TA}_{:,:,i}]) $  \% note: rank $\widehat{\TA}_{:,:,i}$ is $\rho_i$. \\
\ENDFOR
\STATE  $\TU = \widehat{\TU} \times_3 \MM^{-1}, \,\, \TS = \widehat{\TS} \times_3 \MM^{-1}, \,\, \TV = \widehat{\TV} \times_3 \MM^{-1}$. \\
\end{algorithmic}
\end{algorithm}

Independent of the choice of $\MM$, the components of the t-SVDM are computed in transform space.  We describe the full t-SVDM in \Cref{alg:one}.
As noted, the t-SVDM above was proposed already in \cite{LAA}.  However, when we restrict the class of $\MM$ to non-zero multiples of unitary or orthogonal matrices, we can now 
derive an Eckart-Young Theorem for tensors in general form. 
To do so, we first give a new corollary for the
restricted class of $\MM$ considered. 
\begin{corollary}
Assume $\MM = c \M{W}$, where $c \not= 0$, and $\M{W}$ is unitary.  Then given the t-SVDM of $\TA$ over $\starM$ defined in \Cref{def:tsvd},  
\[ \| \TA \|_F^2  = \| \TS \|_F ^2 = \sum_{i=1}^{\min(p,m)} \| \TS_{i,i,:} \|_F^2. \]  
Moreover, $\| \TS_{1,1,:} \|_F^2 \ge \| \TS_{2,2,:} \|_F^2 \ge \ldots $
\end{corollary}
\begin{proof}  The proof of the first equality follows directly from \Cref{th:invariance}, the second from the definition of Frobenius norm.   To prove the ordering property, use the short hand for each singular tube fiber as $ \V{s}_i := \TS_{i,i,:}$, and note using (\ref{eq:norm}) that 
\[  \| \V{s}_i \|_F^2 = \frac{1}{c} \| \hat{\V{s}}_i \|_F^2 = \frac{1}{c} \sum_{j=1}^n (\hat{\sigma}_i^{(j)})^2, \]
where we have used $\hat{\sigma}_i^{(j)}$ to denote the $i^{th}$ largest singular value of the $j^{th}$ frontal face of $\widehat{\TS}$.  
However since $\hat{\sigma}_i^{(j)} \ge \hat{\sigma}_{i+1}^{(j)}$, the result follows. \end{proof} 

This decomposition and observation gives rise to a new definition
\begin{definition}
We refer to $r$ in the t-SVDM \Cref{def:tsvd} (see the second equality in \Cref{eq:TSVDMr}) as the {\bf t-rank}\footnote{The term t-rank is exclusive to the $\starM$ tensor decomposition and should not be confused with the rank of a tensor which was defined in the introduction. }, the number of non-zero singular tubes in the t-SVDM.  
\end{definition}
We can also extend the idea of multi-rank in \cite{KilmerEtAl2013} to the $\starM$ general case: 
\begin{definition}
The {\bf multi-rank} of $\TA$ under $\starM$ is the {\it vector} $\boldsymbol \rho$ such that its $i^{th}$ entry $\rho_i$ denotes the rank of the $i^{th}$ frontal slice of $\widehat{\TA}$ (See comment in \Cref{alg:one}). 
\end{definition} 
Notice that a tensor with multi-rank $\bfrho$ must have t-rank equal to $\max_{i=1,\ldots,n} \rho_i$.  
\begin{definition} The {\bf implicit rank} under $\starM$ of $\TA_{\bfrho}$ is $r = \sum_{i=1}^{n} \rho_i$. \end{definition}

\subsection{Eckart-Young Theorem for Tensors} 

The key aspect that has made the t-product-based t-SVD so instrumental in many applications (see, for example \cite{HaoEtAl2013,zhang2014novel,Zhang2017MultiviewSC,SAGHEER2019}) is the tensor
Eckart-Young theorem proven in \cite{KilmerMartin2011} for real valued tensors under the t-product.   
In loose terms, {\it truncating the t-product based t-SVDM gives an optimal low t-rank approximation in the Frobenius norm}. 
An Eckart-Young theorem for the $\starM$ operator was not provided in \cite{LAA}.  We give a proof below for the special case that we have been considering in which $\MM$ is a multiple of a unitary matrix.   
\begin{theorem} 
 Define $\TA_k = \TU_{:,1:k,:} \starM \TS_{1:k,1:k,:} \starM \TV_{:,1:k,:}^{\rm H}$, where $\MM$ is a non-zero multiple of a unitary matrix.   Then $ \TA_k$ is the best Frobenius norm approximation\footnote{By ``best" we mean minimizer of the Frobenius norm of the discrepancy between the original and the approximated tensor.} over the set $\Gamma = \{ \TC = \TX \starM \TY | \TX \in \mathbb{C}^{m \times k \times n}, \TY \in \mathbb{C}^{k \times p \times n} \}, $ the set of all t-rank $k$ tensors under $\starM$, of the same dimensions as $\TA$.  The squared error is 
$\| \TA - \TA_k \|_F^2 = \sum_{i=k+1}^{r} \| \V{s}_i \|_F^2 $, where $r$ is the t-rank of $\TA$.    
\end{theorem}
\begin{proof} The squared error result follows easily from the results in the previous section.  Now let $\TB = \TX \starM \TY$.  $\| \TA - \TB \|_F^2 = \frac{1}{c} \| \widehat{\TA}  - \widehat{\TB} \|_F^2 = \frac{1}{c} \sum_{i=1}^n \| \widehat{\TA}_{:,:,i} - \widehat{\TB}_{:,:,i} \|_F^2. $   By definition, $\widehat{\TB}_{:,:,i}$ is a rank-k outer product 
$ \widehat{\TX}_{:,:,i} \widehat{\TY}_{:,:,i}$.   The best rank-k approximation to $\widehat{\TB}_{:,:,i}$ is $\widehat{\TU}_{:,1:k,i} \widehat{\TS}_{1:k,1:k,i} \widehat{\TV}_{:,1:k,i}^{\rm H}$, so 
$ \| \widehat{\TA}_{:,:,i} - \widehat{\TU}_{:,1:k,i} \widehat{\TS}_{1:k,1:k,i} \widehat{\TV}_{:,1:k,i}^{\rm H} \|_F^2 \le 
\| \widehat{\TA}_{:,:,i}  - \widehat{\TB}_{:,:,i} \|_F^2$, and the result follows.  
  \end{proof}
%
%

In \cite{HaoEtAl2013}, the authors used the Eckart-Young result for the t-product for compression of facial data and a PCA-like approach to recognition.   They also gained additional compression in an algorithm they called the t-SVDII (only for the t-product on real tensors), which, although not described as such in that paper, is effectively reducing the multi-rank for further compression.  
Here, we provide the theoretical justification for the t-SVDII approach in \cite{HaoEtAl2013} while simultaneously extending the result to the $\starM$ product family restricted to $\MM$ being a non-zero multiple of a unitary matrix.   
\begin{theorem} \label{thm:tsvdMII}
Given the t-SVDM of $\TA$ under $\starM$, define $\TA_{\bfrho}$, to be the approximation having multi-rank $\bfrho$: that is, 
$$ (\widehat{\TA}_{\bfrho})_{:,:,i} = \widehat{\TU}_{:,1:\rho_i,i} \widehat{\TS}_{1:\rho_i,1:\rho_i,i} \widehat{\TV}_{:,1:\rho_i,i}^{\rm H}. $$
Then $\TA_{\bfrho}$ is the best multi-rank $\bfrho$ approximation to $\TA$ in the Frobenius norm and 
\[ \| \TA - \TA_{\bfrho} \|_F^2 = \sum_{i=1}^n \sum_{k=1}^{r_i} (\hat{\sigma}_{\rho_i+k}^{(i)})^2, \]
where $r_i$ denotes the rank of the $i^{th}$ frontal face of $\widehat{\TA}$. 
\end{theorem}
   \begin{proof}  Follows similarly to the above, and is omitted. \end{proof}

To use this in practice, we generalize the idea of the t-SVDII in \cite{HaoEtAl2013} to the $\starM$ product when $\MM$ is a multiple of a unitary matrix.  First, we need a suitable method to choose $\bfrho$.  We know
 \[ \| \widehat{\TA} \|_F^2 = \sum_{i=1}^{n} \sum_{j=1}^{r_i} (\hat{\sigma}^{(i)}_j)^2 ,\]
where $r_i$ is the rank of the $i^{th}$ frontal face of $\widehat{\TA}$.  Thus, there are $K:=\sum_{i=1}^{n} r_i \le  n\min(m,p)$ total non-zero singular values.  
Let us order the $(\hat{\sigma}^{(i)}_j)^2$ values in descending order, put them into a vector of length $K$.  We find the  first index $ J \le K$ such that $ (\sum_{i=1}^J v_i)/ \| \widehat{\TA} \|_F^2  > \gamma$.  Keeping $J$ total terms thus implies an approximation of energy $\gamma$.  Then let $\tau = \sqrt{ v_J}$ -- this will be the value of the singular value that is the smallest one which we should include in the approximation.   We run back through the $n$ faces, and for face $i$, we keep only the $\rho_i$ singular 3-tuples
such that $\hat{\sigma}_{j}^{(i)} \ge \tau$.  In other words, the relative error 
in our approximation is given by
 \[ \frac{ \sum_{i=1}^n \sum_{j=1}^{\rho_i} (\hat{\sigma}_{j}^{(i)})^2 }{\| \widehat{\TA} \|_F^2} \approx \gamma \]
The pseudocode is given in \Cref{alg:tsvdMII}, and a cartoon illustration of the output is given in \Cref{fig:tsvdII}.
  
\begin{algorithm} 
\caption{\label{alg:tsvdMII} Return t-SVDMII under $\starM$, $\bfrho$ to meet energy constraint.}
\begin{algorithmic}[1]
\STATE INPUT:  $\TA$, $\MM$ a multiple of unitary matrix; desired energy $\gamma \in (0,1]$. \\
\STATE Compute t-SVDM of $\TA$.   \\
\STATE Concatenate $((\widehat{\TS}_{j,j,i}).^2)$ for all $i,j$ into a vector $\V{v}$. \\
\STATE $\V{v} \leftarrow \mbox{sort}(\V{v},'\mbox{descend}')$.  \\
\STATE Let $\V{w}$ be the vector of cumulative sums: i.e. $\V{w}_k = \sum_{i=1}^k \V{v}_i$ \\ 
\STATE Find the first index $J$ such that $\V{w}_J/ \| \hat{\TS} \|_F^2 > \gamma$.  \\
\STATE Define $\tau := \V{v}_J$.  
\FORALL{$i = 1,\ldots, n$}
\STATE Set $\rho_i$ as number of singular values for $\widehat{\TA}_{:,:,i}$ greater or equal to $\tau$. \\
\STATE Keep only the $m \times \rho_i$ $\widehat{\TU}_{:,1:\rho_i,i}$ and $\widehat{\TG}_\bfrho :=\widehat{\TS}_{1:\rho_i,1:\rho_i,i} \widehat{\TV}_{:,1:\rho_i,i}^{\rm H}$.\\ 
\ENDFOR
\end{algorithmic}
\end{algorithm} 


\begin{figure*}[h]
\centering
\includegraphics[scale=.9]{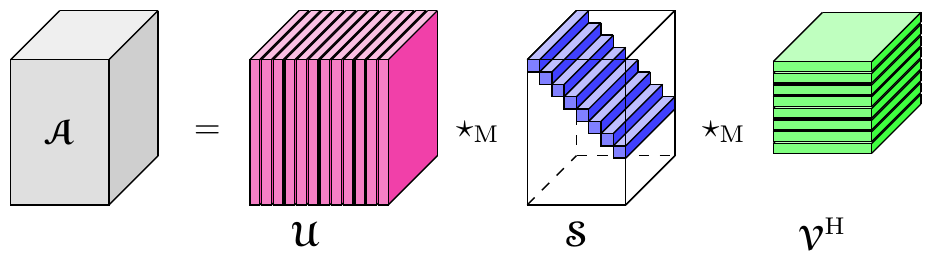}
\includegraphics[scale=.9]{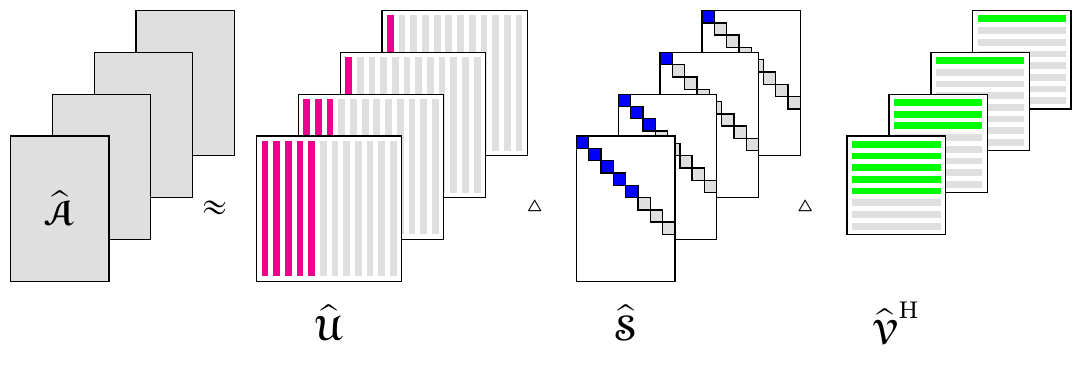}
\caption{\label{fig:tsvdII}Top: Illustration of the tensor SVD. Bottom: Example showing different truncations across the different SVDs of the faces, based on \Cref{alg:tsvdMII}. }
\end{figure*}

 
In  \Cref{sec:superior}, we compare our Eckart-Young results for tensors with the corresponding matrix approximations obtained by using the matrix-based Eckart-Young theorem.  But first, we need a few results that show what structure these tensor approximations inherit from $\MM$.

\section{Latent Structure}\label{sec:latent}
To understand why the proposed tensor decompositions are efficient at compression and feature extraction, we investigate the latent structure 
induced by the algebra in which we operate.  We shall also capitalize on this structural analysis in the the next section's proofs stating how the proposed t-SVDM and t-SVDMII decompositions can be used to devise superior approximations compared to their matrix counterparts.   

If $\V{v}, \V{c} \in \mathbb{C}^{1 \times 1 \times n}$, from  \Cref{alg:prod} we have
$$ \V{v} \starM \V{c} = ((\V{v} \times_3 \MM) \odot ((\V{c} \times_3 \MM))\times_3 \MM^{-1}, $$
where $\odot$ indicates pointwise scalar products on each face. 
Using the aforementioned definitions, this expression is tantamount to
\begin{equation}
 \label{eq:mytwist} 
 \mbox{\tt twist}[\left( ( \MM^{-1} \mbox{\tt diag}(\MM \, \V{c}_{(3)}) \MM \, \V{v}_{(3)} \right)^{\top}]  = 
 \mbox{\tt twist}[ {\V{v}_{(3)}}^\top \MM^\top \mbox{\tt diag} (\hat{\V{c}}) \MM^{-\top}] \end{equation}
 Note that $\MM \V{c}_{(3)}$ is mathematically equivalent 
to forming the tube fiber $\hat{\V{c}}$, and 
$\tt diag$ applied to a tube fiber works analogously to $\tt diag$ applied to that tube fiber's column vector equivalent.  Further, the transpose is a real transpose, stemming from the definition of the mode-3 product. 

Let $\vec{\T{B}}$ be any element in 
$\mathbb{C}^{m \times 1 \times n}$ 
and consider computing the product $\vec{\TQ} : = \vec{\T{B}} \starM \V{c}$.  In \cite{LAA} it was shown that the $j^{th}$ tube fiber entry in $\vec{\TQ}$ is effectively the product of the tubes $\vec{\T{B}}_{j,1,:} \starM \V{c}$.   From (\ref{eq:mytwist}) we have
\begin{equation}  \label{eq:flat} \sq{\vec{\TQ}} = \sq{\vec{\TB}} \left( \MM^{\top} \mbox{\tt diag}( \hat{\V{c}} ) \MM^{-\top} \right). \end{equation}   The matrix in the parenthesis on the right is an element of the space of all matrices described 
$ \mathcal{X}_{\MM} = \{ \M{X} : \M{X} = \MM^{\top} \M{D} \MM^{-\top} \}$, where $\M{D}$ is a diagonal matrix.   
This realization brings us to a major result.  
\begin{theorem} Suppose that 
$\TA  =  \T{U} \starM \underbrace{\T{S} \starM \T{V}^{\rm H}}_{\T{C}}  =  \sum_{i=1}^{t} \T{U}_{:,i,:} \starM \T{C}_{i,:,:} .$
Then
\begin{equation} \label{eq:latslice} \sq{\TA_{:,k,:} }  = \sum_{i=1}^{t} \sq{\T{U}_{:,i,:} } \MR[\T{C}_{i,k,:}], 
                                                    =  \sum_{i=1}^{t} \M{U}_i \MR[\T{C}_{i,k,:}].  ,\end{equation}
where $\MR[ \V{v} ] := \MM^{\top} \mbox{\rm diag}( \hat{\V{v}} ) \MM^{-\top}$ 
maps a tube fiber to a matrix in the set $\mathcal{X}_\MM$. 
\end{theorem}

Thus, each {\bf lateral} slice of $\TA$ is a weighted combination of ``basis'' matrices given by 
$ \M{U}_i  := \sq{ \TU_{:,i,:} }$, but the weights, instead of being scalars, are matrices $\MR[\T{C}_{i,k,:}]$ from the matrix algebra induced by the choice of $\MM$.  
For $\MM$ the DFT matrix, the matrix algebra is the algebra of circulants.

\section{Tensors and Optimal Approximations}  \label{sec:superior}

In \cite{HaoEtAl2013}, the claim was made that a t-SVD to $k$ terms could be superior to a matrix-SVD based compression to $k$ terms.   Here, we offer a formal proof, then discuss the relative meaning of $k$.   Then in the next section, we discuss what can be done to obtain further compression.

\subsection{Theory:  T-rank vs. Matrix Rank}


Let us assume that our data is a collection of $\ell$, $m \times n$ matrices $\MD_i, i=1,\ldots \ell$.  For example, $\MD_i$ might be a gray scale image, or it might be the values of a function discretized on a 2D uniform grid.  Let $\V{d}_i = \mbox{\tt vec}(\MD_i)$, so that $\V{d}_i$ has length $mn$.  

We put samples into a matrix (tensor) from left to right:
\[ \MA = \left[ \V{d}_1,\ldots,\V{d}_\ell \right] \in \mathbb{C}^{mn \times \ell} \qquad
\TA = \left[ \twist{\MD_i},\ldots, \twist{\MD_\ell} \right] \in \mathbb{C}^{m \times \ell \times n}. \]
Thus, $\MA, \TA$ represent the same data, just in different formats.
It is first instructive to consider in what ways the t-rank, $t$, of $\TA$ and the matrix rank $r$ of $\MA$ are related.  Then we will move on to relating the optimal t-rank $k$ approximation of $\TA$ with the optimal rank-$k$ approximation to $\MA$.   

 \begin{theorem}  \label{thm:trank} 
The t-rank, $t$, of $\TA$ is less than or equal to the rank, $r$, of $\MA$.  Additionally, since $t \le \min(m,\ell)$, if $m < r$, then $t < r$.   
    \end{theorem}  
\begin{proof}
The problem dimensions necessitate $t \le \min(m, \ell)$ and $r \le \ell$.  
Let $\MA = \M{G} \M{H}^\top$ be a rank-r factorization of $\MA$ such that $\MG$ is $mn \times r$ and $\M{H}^\top$ is $r \times \ell$. 
From the fact $\widehat{\TA} = \TA \times_3 \MM$, we can show the $m \times \ell$ sized $i^{th}$ frontal face of $\widehat{\TA}$ satisfies 
\begin{equation} \label{eq:tall} 
\widehat{\TA}_{:,:,i} = \sum_{j=1}^n m_{ij} \TA_{:,:,j} = \sum_{j=1}^n m_{ij} \M{G}_{(j-1)m+jm,:} \M{H}^\top
= (\sum_{j=1}^n m_{ij} \M{G}_{j-1)m + jm,:}) \M{H}^\top. \end{equation}
Clearly, the rank of this frontal slice is bounded above by $\min(m,r)$ since this is the maximal rank of the matrix in parenthesis.  
Then the singular values of the matrix $\widehat{\TA}_{:,:,i}$, satisfy $\hat{\sigma}_1^{(i)} \ge \hat{\sigma}_{2}^{(i)} \ge \dots \hat{\sigma}_{r_i}^{(i)}$, where $r_i \le \min(m,r)$.    
As $\widehat{\TS}_{j,j,i} = \hat{\sigma}_j^{(i)} $, 
 $\TS_{j,j,:} = \widehat{\TS}_{j,j,:} \times_3 \MM^{-1}$ for a particular value $j$ will be a non-zero tube fiber iff for any of the $i=1:n$, at least one $\hat{\sigma}_j^{(i)}$ is non-zero.  There can be at most $\min(m,r)$ non-zero tube fibers, so $t \le \min(m,r)$. 
 \end{proof}
  
Note that the proof was {\bf independent of the choice of $\MM$ as a multiple of a unitary matrix}.  In particular, it holds for $\MM = \M{I}$.  This means that simply the act of 'folding' the data matrix into a tensor may provide a reduced rank approximation (a rank-$r$ matrix goes to a t-rank $< r$ tensor).   Under a non-identity choice of $\MM$, though may reveal 
$t \ll r$.   
To make the idea concrete, let us consider an example in which the t-rank under $\starM$ for $\MM$ the DFT matrix is 1, but for which the matrix case does not reveal this structure. 
%
%
%

\begin{example} \label{ex:superiorB}
Let $\MU \in \mathbb{R}^{n\times n}$ invertible, and $\V{c}_i \in \mathbb{R}^{n}$, $i=1,\ldots,p$ with $p \le n$
be a set of independent vectors. Define $\TA$ such that $\TA_{:,i,:} = \mbox{\tt twist}(\MU \mbox{\tt circ}(\V{c}_i))$. 
Is is easy to see that the t-rank is 1.  

On the other hand, with $\MZ$ the circulant downshift matrix,
\[ \MA = (\M{I} \otimes \MU) \mbox{\rm{diag}}(\M{I}, \MZ,\ldots, \MZ^{n-1}) \bea \V{c}_1 & \V{c}_2 & \cdots & \V{c}_p \\
\V{c}_1 & \V{c}_2 & \cdots & \V{c}_p \\
\vdots & \vdots & \vdots & \vdots \\
\V{c}_1 & \V{c}_2 & \cdots & \V{c}_p \eea \]
The rank of the $\MA$ therefore is $p$.   Indeed, $\MA$ can be highly incompressible:  if $\MU$ and $\M{C}$ have orthonormal columns, then we can show $\| \MA - \MA_k \|_F^2 = (p-k)n$ for any $k < n$.  

\end{example}

 \subsection{Theory:  Comparison of Optimal Approximations}      
In this subsection, we want to compare the quality of approximations obtained by truncating the matrix SVD of the data matrix vs. truncating the t-SVDM of the same data as a tensor.  
In what follows, we again assume that $m n > \ell$ and that $\MA$ has rank $r \le \ell$.   

Let $\MA = \MU {\bf \Sigma} \MV^\top$ be the matrix SVD of $\MA$, and denote its best rank-$k$, $k < r$  approximation according to
\begin{equation} \label{eq:akcol} \M{C} = {\bf \Sigma} \MV^\top,  \qquad \MA_k := \MU_{:,1:k} \M{C}_{1:k,:} \Rightarrow (\MA_k)_{:,j} = \sum_{i=1}^k \MU_{:,i} c_{ij} . \end{equation} 
Lastly, we need the following matrix version of (\ref{eq:akcol}) which we reference in the proof:
        \begin{equation} \label{eq:cols}  \mbox{\tt reshape}( (\M{A}_k)_{:,j},[m,n]) = \sum_{i=1}^k \mbox{\tt reshape}(\M{U}_{:,i},[m,n]) c_{ij}, \qquad j=1,\ldots,k. \end{equation}

\begin{theorem} \label{th:tsvdopt} Given  $\M{A}$, $\TA$ as defined above, with $\M{A}$ having rank $r$ and $\TA$ having t-rank $t$, let $\M{A}_k$ denote 
the best rank-$k$ matrix approximation to $\M{A}$ in the Frobenius norm, where $k \le r$.  Let $\TA_k$ denote the best t-rank-$k$ tensor approximation under $\starM$, where $\MM$ is a multiple of a unitary matrix, to $\TA$ in the Frobenius norm.    
 Then 
\[  
\| \TS_{k+1: t,k+1: t,: } \|_F = \| \TA - \TA_k \|_F \le \| \MA - \MA_k \|_F .
\]
 \end{theorem}
 

\begin{proof} 
Consider (\ref{eq:cols}).  The multiplication by the scalar $c_{ij}$ in the sum is equivalent to multiplication from the right by $c_{ij} \M{I}$.
However, since $\MM = c \M{W}$ for unitary $\M{W}$, we have $c_{ij} \M{I} = \MM^{\top} \mbox{diag}( c_{ij} \V{e} ) \MM^{-\top}$, where $\V{e}$ is the vector of all ones.   Define the tube fiber $\TC_{i,j,:}$ 
from the matrix-vector product $c_{ij} \MM^{-1} \V{e}$ oriented into the 3rd dimension.   Then, $c_{ij} \M{I} = \MR[\TC_{i,j,:}]$.   
Now we observe that (\ref{eq:cols})
can be equivalently expressed as 
\begin{equation} \label{eq:colstwo} \mbox{\tt reshape}((\M{A}_k)_{:,j}) = \sum_{i=1}^k \mbox{\tt reshape}(\M{U}_{:,i}) \MR[ \V{c}_{ij}], \qquad j=1,\ldots,k.\end{equation}
These can be combined into a tensor equivalent
\[ \T{Z}_k := \sum_{i=1}^k \T{Q}_{:,i,:} \starM \T{C}_{i,:,:} = \TQ \starM \TC \qquad \mbox{ where } \]
 \[ (\T{Z}_k)_{:,j,:} = \twist{\mbox{\tt reshape}((\MA_k)_{:,j},[m,n])}, \qquad \TQ_{:,i,:} = \twist{\mbox{\tt reshape}(\MU_{:,i},[m,n])}. \]
Since $\widehat{\TC}_{:,:,i} = \M{\Sigma}_{1:k,1:k} \M{V}_{:,1:k}^\top$, the t-rank of $\TC$ is $k$.   The t-rank of $\TQ$ must also not be smaller than $k$, by \Cref{thm:trank}.
  
Thus, given the definition of $\T{A}_k$ as the minimizer over all such $k$-term `outer-products' under $\starM$, it follows that
\[
 \| \T{S}_{k+1:t,k+1:t,:} \|_F  =   \| \TA - \T{A}_k \|_F 
                                            \le  \| \TA - \T{Z}_k \|_F  
                                             =  \| \MA - \MA_k  \|_F .  \]
     \end{proof} 

Here is one small example showing strict inequality is possible.
Additional supporting examples are in the numerical results. 
%

\begin{example}
Given $\MM$ the DFT matrix, and let 
$$\MA = \bea 1 & 1 \\ 1 & 4 \\ 0 & 0 \\ 0 & -3 \eea,  \qquad \mbox{with } \MA_1 = \sigma_1 \Vu_1 \Vv_1^{\top}.$$   It is easily verified that $\| \MA - \MA_1 \|_F = \sqrt{\sigma_2} = 1$. 
It is easy to show 
\[  \hat{\TA}_{:,:,1} = \bea 1 & 1 \\ 1 & 1 \eea 
%
\qquad \hat{\TA}_{:,:,2} = \bea 1 & 1 \\ 1 & 7 \eea
, \]
Setting $\TA_1 = \TU_{:,1,:} \starM \TS_{1,1,:} \starM \TV_{:,1,:}^{\top}$, then
\[   \| \TA - \TA_1 \|_F  =  \| \TS_{2,2,:} \|_F 
                               =  \left\| \M{M}^{\rm H} \bea 0 \\ \hat{\sigma}_2^{(2)} \eea \right\|_F 
                               =  \frac{1}{\sqrt{2}} \hat{\sigma_2}^{(2)} \approx .59236 < 1. \]
\end{example}

In the next subsection we discuss the level of approximation provided by the output of \Cref{alg:tsvdMII} by relating it back to the truncated t-SVDM and also to truncated matrix SVD.  First, we need a way to relate storage costs. 


\begin{theorem} \label{thm:doubleopt} Let $\TA_k$ be the t-SVDM t-rank $k$ approximation to $\TA$, and suppose its implicit rank is $r$. 
Define $\mu = \|\TA_k \|_F^2/\|\TA \|_F^2$.
There exists $\gamma \le \mu$ such that the
t-SVDMII approximation, $\TA_{\bfrho}$, obtained for this $\gamma$ in \Cref{alg:tsvdMII}, has implicit rank less than or equal to an implicit rank of $\TA_k$ and
\[ \| \TA - \TA_{\bfrho} \|_F \le \| \TA - \TA_k \|_F \le \| \MA - \MA_k \|_F. \]
\end{theorem}
\begin{proof}
From \Cref{th:tsvdopt} that $\| \TA - \TA_k \|_F^2 = \frac{1}{c} \sum_{j=1}^n \sum_{i=k+1}^n (\hat{\sigma}_i^{(j)})^2$.  The proof is by construction using $\TA_k$ as the starting point.  

Set $\rho_i = k$ to start.   
 Let $C = \{ \hat{\sigma}_{j}^{(i)} < \sigma_* | 1 \le j \le \rho_i, \forall i \}$ 
where 
\[ \hat{\sigma}_* = \max_{i=1,\ldots,n} \sigma_{\rho_i+1}^{(i)}. \]  
In other words, we look at the union of the singular values in each face of $\widehat{\TA}$ that were omitted from the current approximation, to see if there is at least one that is larger than those that were included in the approximation.   If the set is 
empty, then $\TA_{\bfrho} = \TA_k$ and we are done.   

Otherwise, the $i*^{th}$ face has a singular value that is larger than the singular values that were included in the approximation $\TA_k$.
For convenience, label elements of $C$ such that they are in increasing order ($c_1 \le c_2 \le c_3...$).  
Define $\pi_p = \sum_{i=1}^{p} c_i^2$, and $p$ is less than or equal to the cardinality of $C$.  There must exist 
at least one value of $p$ such that $\hat{\sigma}_{i*}^2 > \pi_p$.  
Set $\rho_{i*} \leftarrow  \rho_{i*}+1$, 
and reduce the $p$ values of the $\rho_i$ that correspond to $c_1,\ldots,c_p$.

Then the error $\| \TA - \TA_\bfrho \|_F^2$ has been decreased by an amount $\sigma_*^2$ while the error is simultaneously increased by $\sum_{k=1}^p  c_k^2$.  We can take $p$ as large as possible so that the relative increase in the error keeps the total below $\mu$.  
Thus, the implicit rank has decreased by $p-1$ but error remains bounded by $\mu$.  This process can be repeated until
we reach an iteration when $C$ is empty.   
\end{proof}



\subsection{Storage Comparisons}
Let us suppose that $\kappa$ is the truncation parameter for the tensor approximation and $k$ is the truncation parameter for the matrix approximation.   \Cref{tab:storeone} gives a comparison of storage for the methods we have discussed so far.  Note that for the t-SVDMII, it is necessary to work only in the transform domain, as moving back to the spatial domain would cause fill and unnecessary storage.  We often use $\MM$ that can be applied using fast transform techniques (such as the DCT, DFT, or discrete wavelet transform), so we do not include 
storage costs associated with $\MM$.  Storage of $\MM$ is discussed further in \Cref{ssec:newsto}.
\begin{table}[h] \label{tab:storeone}
\begin{tabular}{ccc}
Storage for basis $\MU_k$ & Storage for $C = \MS_k \MV_k^{\rm H}$ & total implicit storage $\MA_k$ \\ \hline
  $kmn$                   & $kp$   & $k(mn + p)$ \\ \hline \hline 
Storage for basis $\TU_\kappa$ & Storage for $\TC = \TS_\kappa \starM \TV_\kappa^{\rm H}$ & total implicit storage $\TA_\kappa$ \\ \hline
  $ \kappa mn $                 &  $\kappa p n$ & $ \kappa mn + \kappa p n$\\ \hline
Storage for $\widehat{\TU}_{\bfrho}$  & Storage for $\widehat{\TC}_{\bfrho}$ & total implicit storage $\TA_\bfrho$ \\ \hline
$m r $ & $rp$  & $ (m+p)r  $\\ \hline
\end{tabular}
\caption{Storage costs using $k$-term truncated SVD expansion vs. $\kappa$-term truncated t-SVDM expansion vs. t-SVDMII.  Recall that for the $t-SVDMII$, we store terms in the transform domain.  Recall the implicit rank for $\TA_\bfrho$ is $r = \sum_{i=1}^n \rho_i$.}
\end{table}

\paragraph{Discussion}
If $\kappa = k$, the theorem says approximation error is at least as good than the corresponding matrix approximation.  In applications where we only need to store the basis terms, e.g. to do projections, {\it the basis for the tensor approximation is better in a relative error sense than the basis for the matrix case, for the same storage}.    However, unless $n=1$, if we need to store both the basis {\it and the coefficients}, we will need more storage for the tensor case {\it if} we need to take $\kappa = k$.    
Fortunately in practice, $\| \TA - \TA_{\kappa} \|_F \le \| \MA - \MA_k \|_F$ for $\kappa < k$.  Indeed, we already showed an example where the error is zero for $\kappa = 1$, but $k$ had to be much larger to achieve exact approximation.  If $\frac{\kappa}{k} < \frac{m + \frac{p}{n}}{m + p}$, then the total implicit storage of the tensor approximation of $\kappa$ terms is less than the total storage for the matrix case of $k$ terms.  

Compared to the matrix SVD, the t-SVDMII approach can provide compression for at least as good, or better, an approximation level as indicated by the theorem.  Of course, $\MM$ should be an appropriate one given the latent structure in the data.
The t-SVDMII approach allows us to account for the more “important” features (e.g. low frequencies, multidimensional correlations), and therefore impose a larger truncation on the corresponding frontal faces because those features contribute more to the global approximation.   Truncation of t-SVDM by a single truncation index, on the other hand, effectively treats all features equally, and always truncates each $\widehat{\TA}_{:,:,i}$ to $k$ terms, which depending on the choice of $\MM$ may not be as good.    This is demonstrated in the numerical results section.


\section{Comparison to Other Tensor Decompositions} \label{sec:hosvd}

In this section, we compare to the two other types of tensor representations described in the introduction:  truncated HOSVD and CP types of decomposition.

\subsection{Comparison to tr-HOSVD}
In
this section, we wish to show how tr-HOSVD can be expressed using a $\starM$ product.  Then we can compare our truncated results to the tr-HOSVD.   


The truncated HOSVD (tr-HOSVD) is formed by truncating to $k_1, k_2, k_3$ columns, respectively, the factor matrices $\MQ, \MW, \MZ$, and forming the $k_1 \times k_2 \times k_3$ core tensor as
\[ \TC_{\bf k} : = \TA \times_1 \MQ_{:,1:k_1}^{\top} \times_2 \MW_{:,1:k_2}^{\top} \times_3 \MZ_{:,1:k_3}^{\top} , \]
where $\bf k$ denotes the triple $(k_1,k_2,k_3)$.
The tr-HOSVD approximation then
\[ \TA_{\bf k} =  \TC_{\bf k} \times_1 \MQ_{:,1:k_1} \times_2 \MW_{:,1:k_2} \times_3 \MZ_{:,1:k_3} \]

We now prove the following theorem that shows the tr-HOSVD can be represented under $\starM$ when $\MM = \M{Z}^{\top}$. 
\begin{theorem}  Define the $n \times n$ matrix $\MM$ as $\MM = \M{Z}^{\top}$ (since $\M{Z}$ is unitary, it follows that $\MM^{-1} = \M{Z}$), and define $m \times k_1 \times n$ and $p \times k_2 \times n$ 
tensors in the transform space according to 
\[ \widehat{\TQ}_{:,:,i} = \MQ,   \quad \mbox{and} \quad \widehat{\TW}_{:,:,i} = \MW \quad \text{for} \quad i = 1,\ldots,n .\]
Then $\TQ = \widehat{\TQ} \times_3 \MM^{-1}, \TW = \widehat{\TW} \times_3  \MM^{-1} $ and  
it is easy to show that $\TQ$, $\TW$ are unitary tensors.  Define $\widehat{\TP}$ as the $p \times p \times n$ tensor with identity matrices 
on faces $1$ to $k_3$ and 0 matrices from faces $k_3+1$ to $n$.  
 
Let $ {\TC} = \TQ_{:,1:k_1,:}^{\top} \starM \TA \starM  \TW_{:,1:k_2,:}$.  Then 
\[ \TA_{\bf k} = \TQ_{:,1:k_1,:} \starM {\TC} \starM \TW_{:,1:k_2,:}^{\top} \starM \TP . \]
\end{theorem}

\begin{proof}
First consider  
 \begin{eqnarray}
   \TC  & := & \TA \times_1 \MQ_{:,1:k_1}^\top \times_2 \MW_{:,1:k_2}^\top \times_3 \MZ^\top \\   \nonumber
          & = & (\TA \times_3 \MZ^\top) \times_1 \MQ_{:,1:k_1}^\top \times_2 \MW_{:,1:k_2}^{\top} \\  \nonumber
          & = & \widehat{\TA} \times_1 \MQ_{:,1:k_1}^{\top} \times_2 \MW_{:,1:k_2}^{\top},  \label{eq:hash}\end{eqnarray}
using properties of mode-wise products (see \cite{KoldaBader}).  
From the definitions of the mode-wise product the $i^{th}$ face of $\TC$ as defined via (\ref{eq:hash}) is $\MQ_{:,1:k_1}^{\top} \widehat{\TA}_{:,:,i} \MW_{:,1:k_2}$.   But this means that we can equivalently represent $\TC$ as
$\T{C} = \TQ_{:,1:k_1,:}^{\top} \starM \TA \starM \TW_{:,1:k_2,:}$.   

Now $\TB := \TQ_{:,1:k_1,:} \starM \TC \starM \TW_{:,1:k_2,:}^{\top}$ implies $\widehat{\TB}_{:,:,i} = \MQ_{:,1:k_1} \widehat{\TC}_{:,:,i} \MW_{:,1:k_2}^\top, i=1,\ldots,n$.   But since $\widehat{\TC}_{:,:,i} = \MQ_{:,k_1}^\top \widehat{\TA} \MW_{:,k_2}$ for $i=1,\ldots,n$, we only need to zero-out the last $k_3 + 1:n$ frontal slices of $\TB$ to get to $\TA_{\bf k}$, which we do by taking the $\starM$ product with $\TP$ on the right, and the proof is complete.     
\end{proof}

The value of the theorem is that we now can compare the theoretical results from tr-HOSVD to our truncated methods.   
\begin{theorem} Given the tr-HOSVD approximation $\TA_{\bf k}$, for $\kappa := \min(k_1, k_2)$, 
\[ \| \TA - \TA_\kappa \|_F \le \| \TA - \TA_{\bf k} \|_F \]
with equality only if $\TA_{\bf k} = \TA_{\kappa}$.    \end{theorem}
\begin{proof} Note that the t-ranks of $\TQ$ and  $\TW \starM \TP$ are $k_1$ and $k_2$ respectively.  Since   
${\TC}$ is $k_1 \times k_2 \times n$, its t-rank cannot exceed $\kappa :=\min(k_1,k_2)$.   As such, 
we know $\TA_{\bf k}$ can be written as a sum of $\kappa$ outer-products of tensors under $\starM$, and the result follows given the optimality of $\TA_\kappa$. \end{proof}  

The following now easily follows.  
\begin{corollary} \label{cor:comp} Given the tr-HOSVD approximation $\TA_{\bf k}$, for $\kappa := \min(k_1, k_2)$, there exists $\gamma$ such that $\TA_{\bfrho}$ returned by \Cref{alg:tsvdMII} has implicit rank less than or equal to $\TA_{\kappa}$ and 
\[ \| \TA - \TA_{\bfrho} \|_F \le \| \TA - \TA_{\kappa} \|_F \le \| \TA - \TA_{\bf k} \|_F .\] 
\end{corollary}
Note this is independent of the choice of $k_3$.  Indeed, the size of the upper bound on the right will increase if $k_3 < n$.   

While this has theoretical value, it relies on the choice $\MM = \MZ^\top$.  This begs the question of whether or not $\MZ$ needs to be stored 
explicitly to render the approximation useful in practical applications.  When $\MZ$ is chosen to be a matrix that can be applied quickly without explicit storage, such as a discrete cosine transform, this is not a consideration.  We will say more about this at the end of the section.

\subsection{Approximation in CP Form}  \label{ssec:CP}


Let $\TA_\bfrho$ be given, and let $r$ be the implicit rank.  Define the $m \times r$ matrix $\tilde{\MU}$ by concatenation:
\[  \tilde{\MU} = [\widehat{\MU}_{:,1:\rho_1,1},\widehat{\MU}_{:,1:\rho_2,2}, \cdots,\widehat{\MU}_{:,1:\rho_n,n}] \]
and similarly for $\tilde{\MV}$.   Define $\tilde{\TS}$ to be $r \times r \times n$, and on the $i^{th}$ frontal slice, put the 
entries $\hat{\sigma}_1^{(i)},\ldots,\hat{\sigma}_{\rho_i}^{(i)}$ into diagonal entries numbered $(\sum_{j=1}^{i-1} \rho_i) +1$ to $\sum_{j=1}^{i} \rho_i$.   

Then it is easy to verify that each frontal slice of $\widehat{\TA_{\bfrho}}$ is given by the triple matrix product $\tilde{\MU} \tilde{\MS}_{:,:,i} \tilde{\MV}^{\rm H}$.  
In other words, if the $i^{th}$ row of a matrix $\widehat{\MW}$ contains the diagonal entries of $\widehat{\MS}_{:,:,i}$ starting in column $\sum_{j=1}^{i-1}\rho_i+1$, then
\[\widehat{\TA_{\bfrho}} = \llbracket \tilde{\MU}, \tilde{\MV}, \widehat{\MW} \rrbracket, \]
where all three matrices have $r$ columns.  But using $\tilde{\MW} := \MM^{-1} \widehat{\MW}$,  
\[ \TA_{\bfrho} = \widehat{\TA_{\bfrho}} \times_3 \MM^{-1} = \llbracket \tilde{\MU}, \tilde{\MV}, \tilde{\MW} \rrbracket. \]  

Because there are only $r$ non-zero entries in $\tilde{S}$, only one per tube fiber, each column of $\tilde{\MW}$ is 
$\hat{\sigma}^{(i)}_j (\MM^{-1})_{:,i}$.  If we assume that $\MM$ is an orthogonal matrix, then $\MM^{-1}_{:,i}$ has unit length.  Since we know that each column of $\tilde{\MU}, \tilde{\MV}$ also has unit length, then we know exactly that 
\[ \TA_{\bfrho} = \sum_{k=1}^r \lambda_k \TC_k, \mbox{ where } \TC_k \mbox{ is a rank-1 with $\|\TC_k \|_F=1$} \] 
 and each $\lambda_k = \hat{\sigma}_{j}^{(i)}$ for some $(i,j)$.   We can order the $\lambda$'s so that these are in decreasing order.  Thus, our representation gives a CP formulation of tensor rank at most $r$, and because we have an expression for $\| \TA - \TA_{\bfrho} \|_F^2$ in the theorem, we can chose to truncate this expression further and know precisely the entailed error in
doing so.        

\subsection{Discussion} \label{ssec:newsto} The result tells something important beyond just linking the t-SVDMII to a CP format. The normalized columns of $\tilde{\MW}$ are all multiples of some column of $\tilde{\MM}^{-1}$. Suppose
$\MM = \MZ^\top$.  Then $(\MM^{-1})_{:,i} = \MZ_{:,i}$.  Suppose we fix $\gamma$ and compute our t-SVDMII approximation.  If any $\rho_i = 0$, then
this result means that the $i$th column of $\MZ_{:,i}$ does not appear in $\tilde{\MW}$.  
Implicit storage of $\tilde{\MW}$ requires only those columns of $\MZ$ that appear, and an (integer-valued) array of pointers of length $r$.  Thus, it may be 
possible to use $\MM = \MZ^T$ in practice:  if for our fixed $\gamma$, many $\rho_i = 0$, and/or $n$ is sufficiently small, moving into the transform domain can be accomplished with multiplication by $(\MZ_{\mathcal{J}})^T$, where $\mathcal{J}$ denotes the column indicies for which $\rho_i$ were non-zero.

\section{Multi-sided Tensor Compression}   \label{sec:new}

Given $\TA \in \mathbb{C}^{m \times p \times n}$, we view $\TA$ as an $m \times p$ matrix, with tube-fiber entries of length $n$.   The elemental operation is the $\starM$ operation on the tube-fibers which are then length $n$, and so $\starM$ must be length $n$.

When a data element $\MD_i$ is viewed as an $m \times n$ matrix, and placed into the tensor $\TA$ as a lateral slice, the resulting tensor would have fibers of length $n$. However, in some applications, 
there may be no reason to prefer one spatial ordering over another, but a different ordering can change the size of the third mode, so we need to consider how to treat this case.   

Consider the mapping $\mathbb{C}^{m \times p \times n} \rightarrow \mathbb{C}^{n \times p \times m}$ induced by matrix-transposing (without conjugation) each of the $p$ lateral slices.  
In {\sc Matlab}, this would be obtained by using the command ${\tt permute}(\TA,[3,2,1])$.    To keep the notation succinct, we will use a superscript of $\rm P$ to denote a tensor that has been permuted in this way.   So, 
$$\TA^{\tp} = \mbox{\tt permute}(\TA,[3,2,1]), \qquad (\TA^{\tp})^{\tp} = \TA.$$ 

In this section, we define new techniques for compression that involve both possible orientations of the lateral slices in order to ensure a more balanced approach to the compression of the data.  We use $\starB$ for a tensor-product to operate on the permuted tensors, where $\MB$ is an $m \times m$ non-zero multiple of a unitary matrix. 


 \subsection{Optimal Convex Combinations}
The first option we consider is optimal t-SVDM compression over both orientations. 
      We find the t-SVDM's of both $\TA, \TA^{\tp}$:  $\TA = \TU \starM \TS \starM \TV^{\top} $ and $\TA^{\tp} = \TW \starB \TD \starB \TQ^{\top}$, compress each, and form 
                      \[ \alpha( \TU_{k_1} \starM \TS_{k_1} \starM \TV_{k_1}^{\top}) + 
                      (1-\alpha)*(\TW_{k_2} \starB \TD_{k_2} \starB \TQ_{k_2}^{\top})^{\tp} .\]
                      
   Observe that 
   \[ \mbox{unfold}(\TA^{\tp}) = \M{P} \MA \]
   where $\MA = \mbox{unfold}(\TA)$ as before, and $\M{P}$ denotes a stride permutation matrix.  Since $\M{P}$ is orthogonal, this means that the singular values and right singular vectors of $\MA$ are the same as those of $\M{P} \MA$, and the left singular vectors are row permuted by $\M{P}$.   
   Thus, from our theorem, for a truncation parameter $r$,
   \[ \| \TA^{\tp} - (\TA^{\tp})_r \|_F \le \| \MA - \MA_r \|_F. \]                      
  It follows that
  \begin{eqnarray*}
   \| \TA - \left(\alpha \TA_{k_1} + (1-\alpha) (\TA^{\tp}_{k_2})^{\tp} \right) \| & \le & 
          \| \alpha(\TA - \TA_{k_1}) + (1-\alpha) (\TA - (\TA^{\tp}_{k_2})^{\tp}) \|_F \\ 
          &  \le & 
    \alpha \| \MA - \MA_{k_1} \|_F + (1-\alpha) \| \MA - \MA_{k_2} \|_F \\
          & \le &  \| \MA- \MA_{\min (k_1,k_2)} \|_F . \end{eqnarray*}
  

We can also use the optimal t-SVDMII approximations in convex combination,  
\begin{equation} \label{eq:convexII} \TA_{\bfrho,\bfdel} = \alpha \TA_{\bfrho} + (1-\alpha) (\TA^{\tp}_{\bfdel})^{\tp}, \end{equation}
where $\bfrho, \bfdel$ are the multi-indicies for each orientation, respectively, which may have been determined with
different energy levels.    From \Cref{thm:tsvdMII}, it follows
 \begin{equation} \label{eq:errII} \| \TA - \TA_{\bfrho,\bfdel} \|_F^2 = \alpha^2 \sum_{i=1}^{n} \sum_{j=\rho_i+1}^{r_i}  (\hat{\sigma}_{j}^{(i)})^2
+ (1-\alpha)\sum_{k=1}^{m} \sum_{j=\delta_k+1}^{\tilde{r}_k} (\tilde{\sigma}_{j}^{(k)})^2 ,\end{equation}
where $r_i$ is the rank of $\widehat{\TA}_{:,:,i}, i=1:n$ under $\starM$, $\tilde{r}_k$ is the rank of 
$\widehat{\TA^{\tp}}_{:,:,k}$ under $\starB$ and $\tilde{\sigma}_j^{(k)}$ are for $\widehat{\TA^{\tp}}$ under $\starB$ as well.   
\subsection{Sequential Compression}

Suppose $\TA = \TU \starM \TS \starM \TV^{\top}$ is $m \times p \times n$.  
 Here, as earlier, a subscript on the tensor refers to the number of lateral slices of the tensor that are kept.   For example, $\TU_k$ will denote $\TU_{:,1:k,:}$.   Our new approximation is obtained from the following algorithm.
\begin{algorithm}[h] 
\caption{\label{Alg:Sequential} Sequential t-SVDMB}
\begin{algorithmic}[1]
\STATE INPUT:  $\TA$ of size $m \times p \times n$; truncation parameters $k \le \min(m,p), q \le \min(m,p)$ \nonumber  \\
\STATE INPUT:  $n \times n$ $\MM$ and $k \times k$ $\MB$ that are multiples of unitary matrices. \\
\STATE OUTPUT:  $q \times p \times k$ core $\TG$, left singular tensors $\TU_k$ of $m \times k \times n$ and $\TW_q$ of size $k \times q \times m$ \\ 
\STATE Compute $\TA \approx \TU_k \starM \TS_k \starM \TV_k^{\rm H} $ \\
\STATE Define $k \times p \times n$, $\TC :=  \TU_k^{\top} \starM \TA = \TS_k \starM \TV_k^{\rm H} . $ \\ 
\STATE Find the $q$-term truncated t-SVD of the $n \times p \times k$ tensor $\TC^{\tp}$ under $\starB$:\\
 \[ \TC^{\tp} \approx \TW_{q} \starB \TD_{q} \starB \T{Q}_{q}^{\rm H} . \]
\STATE Use this t-SVD to compress further with a $q$-term projection: \\
 \[  \TG := \TW_q^{\top} \starB \TC^{\tp} = \TD_{q} \starB \T{Q}_q^{\rm H} .\]
 \end{algorithmic}
\end{algorithm} 

 The approximation to $\TA$ is generated via its implicit representation via the triple $(\TG, \TW_q, \TU_k)$ and 
operator pair $(\starB,\starM)$ at a storage cost of $qpk + mkn+kqn$.   Though never explicitly formed, the approximation is 
 \begin{equation} \label{eq:ty}  \TA_{k,q} :=  \TU_k \starM (\TW_q \starB \TG)^{\tp} = \TU_k \starM \left( \TW_q \starB \TW_q^{\rm H} \starB (\TU_k ^{\rm H} \starM \TA)^{\tp} \right)^{\tp} .\end{equation}

 In addition, since the compressed representations are optimal at each stage under their respective algebraic settings, the resulting approximation {\it is locally optimal}:
\begin{theorem} The approximation $\TA_{k,q}$ in (\ref{eq:ty}) above is the best approximation\footnote{Here, we mean it is the minimizer in the Frobenius norm.} that can be expressed in the form
 \[  \sum_{i=1}^{q} \TU_{:,i,:} \starM \left( \sum_{j=1}^{k} \TX_{:,j,:} \starB \T{H}_{j,:,:} \right)^{\tp} \]
in the Frobenius norm, where the lateral slices $\TU_{:,i,:}$ are from the the t-SVDM of $\TA$ and $\TX, \T{H}$ denote any 
$n \times q \times k$ and $q \times p \times k$ tensors, respectively.   
  \end{theorem}

\begin{figure}
\centering
\includegraphics[scale=0.85]{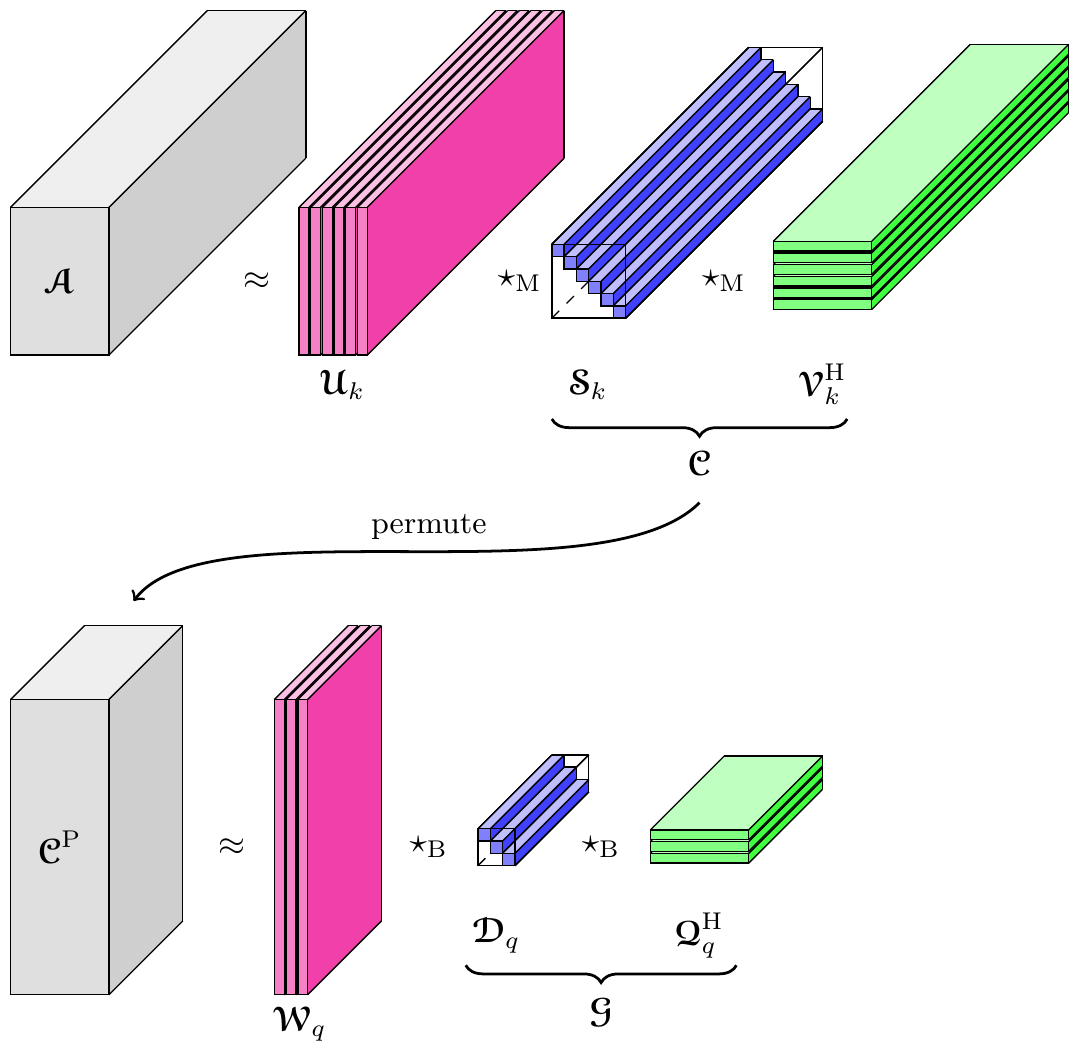}
\caption{\label{fig:multiside} Illustration of the impact of sequential t-SVDM compression \Cref{Alg:Sequential}.}
\end{figure}

Although in this approach we do make use of information on both orientations, the outcome does depend on which orientation is considered first. We can of course apply \Cref{Alg:Sequential} to both $\TA$ and $\TA^{\tp}$, and use a convex combination as the result.   The implicit storage cost is equivalent to storing triples for both approximations.  Assuming the same pair of truncation indicies $(k,q)$ for both orientations,  the total storage for the implicit representation is $2kqp + 2mnk + (n+m)kq$.   Of course, modifications can be made to use different truncation pairs for each sub-problem.   
 
In theory, we can also devise a sequential t-SVDMII if careful consideration is given to keeping computations within transform domains.  However, in our experience, the t-SVDMII is highly compressible as is, and greater benefit is derived from either fixing one orientation, or using a convex combination of both.  Therefore we do not address this further here.

\section{Numerical Examples}   \label{sec:numerical}
                 
In the following discussion, the {\it compression ratio} is defined as the number of floating point numbers needed to store the uncompressed data divided by the number of floating point numbers needed to store the compressed representation (in its implicit form).  Thus the larger the ratio, the better the compression.  The {\it relative error} is the ratio of the Frobenius-norm difference between the original data and the approximation over the Frobenius-norm of the original data. 

\subsection{Compression of YaleB data}

In this section, we show the power of compression for the t-SVDMII approach with appropriate choice of $\MM$:  that is, $\MM$ exploits structure inherent in the data.   We create a third-order tensor from the Extended Yale B face database \cite{yale} by putting the training images in as lateral slices in the tensor.   Then, we apply Algorithm 3, varying $\bfrho$, for 3 different choices of $\MM$:  we choose $\MM$ as a random orthogonal matrix, we use $\MM$ as an orthogonal wavelet matrix, and we use $\MM$ as the unnormalized DCT matrix.   We have chosen to use a random orthogonal matrix in this experiment to show that the compression power is relative to structure that is induced through the choice of $\MM$, so we do not expect, nor do we observe, value in choosing $\MM$ to be random.  In \Cref{fig:tsvdIIcomp}, we plot the compression ratio  against the relative error in the approximation.   We observe that for relative error on the order of 10 to 15 percent, the margins in compression achieved by the t-SVDMII for both the DCT and the Wavelet transform vs. treating the data in either matrix form, or in choosing a transform that - like the matrix case -- does not exploit structure in the data, is quite large.  

\begin{figure}
\begin{center}
\includegraphics[scale=.6]{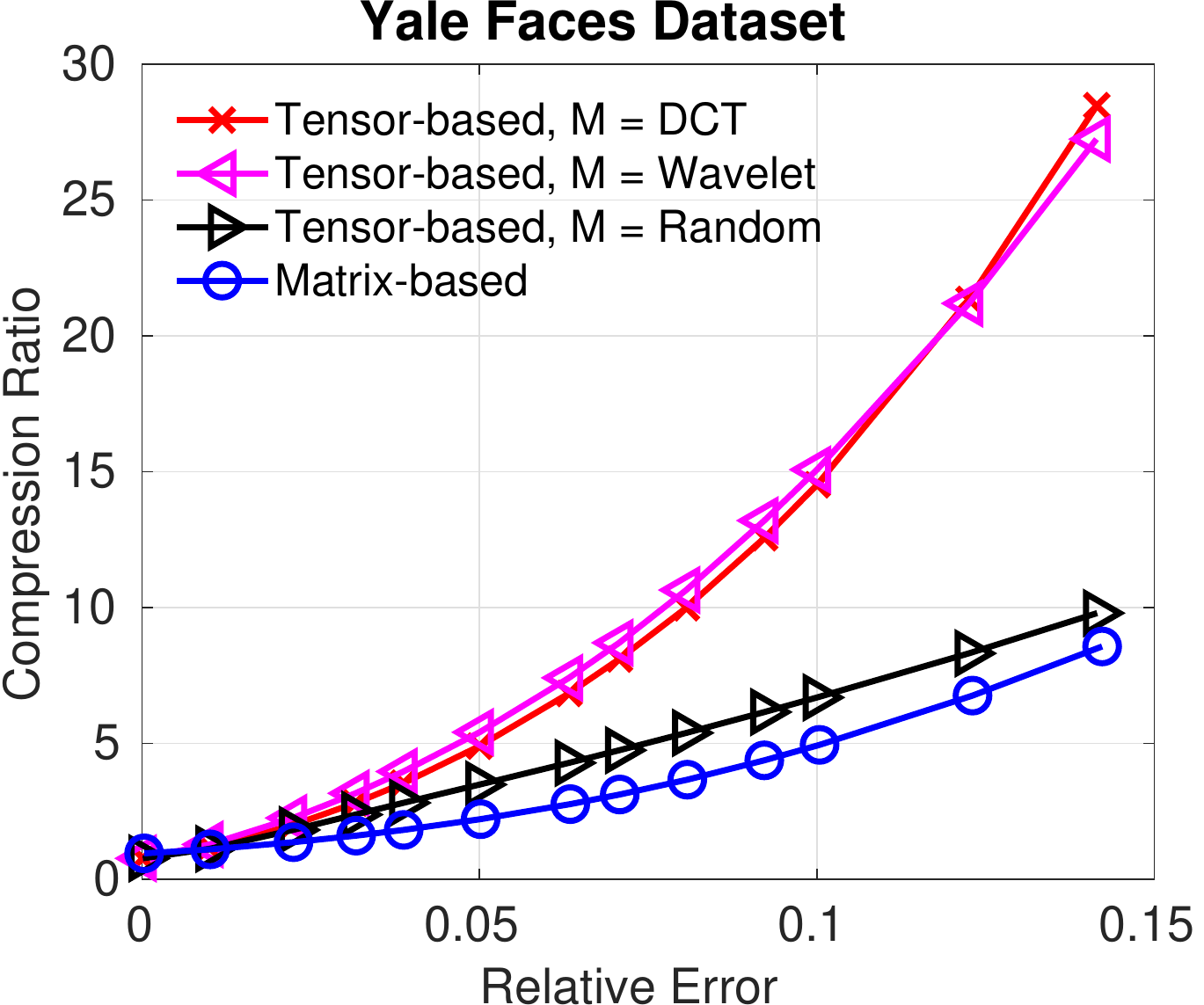}
\caption{\label{fig:tsvdIIcomp}  Illustration of the compressive power of the t-SVDIIM in Algorithm 3 for appropriate choices of $M$. Vertical axis is the inverse of the compression ratio, illustrating that, depending on desired relative error, far more compression is achieved when using either the DCT or wavelet transform to define $\starM$.  This indicates $\MM$ should be chosen to capitalize on structural features in the multiway data, otherwise performance is nearly equivalent to ignoring the multiway structure, as the matrix case does.}
\end{center}
\end{figure}

\subsection{Video Frame Data}

For this experiment, we use video data available in {\sc Matlab}.   The video consists of 120, $120 \times 160$ frames in gray scale.  The camera is positioned near one spot in the road, and cars travel on that road, (more or less from the top to the bottom as the frames progress) so the only changes per frame are cars entering and disappearing.  

We  compare the performance of our truncated t-SVDMII, for $\MM$ being the DCT matrix, against
the truncated matrix and truncated HOSVD approximations.  
We orient the frames as transposed lateral slices in the tensor to try to take advantage of some nearly shift invariant properties of the data\footnote{We could have oriented this with frames into the page as well.  The performance did not vary significantly.}.  Note that in this way, each column of the unfolded data corresponds to a single transposed frame, unfolded (or, equivalently, each column is obtained by unwrapping each frame by rows).   Thus, $\TA$ is $160 \times 120 \times 120$.     

With both the truncated t-SVDMII and the truncated matrix SVD approaches, we can truncate based on the same energy value.  Thus, we get relative error in our respective approximations with about the same value, and then we can compare the relative compression of the truncated t-SVDMII vs. truncated matrix SVD.  On the other hand, we can fix our energy value and compute our truncated t-SVDM2, and find its compression ratio.  Then, we can compute the truncated matrix SVD approximation with similar compression ratio, and compare its relative error to the tensor-based approximation. We will give some results for each of these two ways of comparison.

For truncating the HOSVD to $(k_1,k_2,k_3)$ terms, there are many ways of choosing the truncation 3-tuple. %
Trying to chose a 3-tuple
that has a comparable relative approximation to our approach would be cumbersome and would still leave ambiguities in our selection process.  Thus, rather than looping over all possible choices of the 3-tuple
to find an approximation with relative error closest to our methods, we use two truncation methods and indicies that give us an approximation best matching the compression ratios for our tensor approximation.    In order to make a fair comparison, the indicies are chosen as follows:  1) we compress only on the second mode (i.e. change $k_2$, but fix $k_1=m, k_3 = n$)  2) choose 
the truncation parameters on dimensions such that the mode-wise compression to dimension ratios are about the same.  The second option amounts to looping over $k_2 = 1,\ldots,p$, setting $k_1 = floor( \frac{k_2 m}{n} )$, and $k_3 = k_2$.  In this way, it is possible to compute the compression ratio for the tr-HOSVD based on the dimension in advance to find
the closest match to the desired compression levels.  
The results are given in \Cref{tab:vidresults}.  

\begin{table}
\centering
\begin{tabular}{|c|c|c|c|c|c|} \hline    
 $\gamma_1 = .998$ & t-SVDMII & Matrix on $\gamma_1$ & Mtx on CR & ($m,k_2,n$) & $(k_1,k_2,k_2)$ \\ \hline
CR & 4.76 & 1.83 & 4.76 & 4.95 & 4.90 \\ \hline
RE             & 0.044 & 0.045 & 0.093 & 0.098 & 0.065       \\ \hline \hline
$\gamma_2=.996$ & t-SVDMII & Matrix on $\gamma$ & Mtx on CR & ($m,k,n$) & 
$(k_1,k_2,k_2)$ \\ \hline
CR & 10.10 & 2.54 & 10.87 & 10.75 & 10.42  \\ \hline
RE             & 0.063 & 0.064 & 0.120 & 0.125 & 0.090   \\ \hline 
\end{tabular}
\label{tab:vidresults}
\caption{ Results from video experiments.  CR stands for Compression Ratio and RE stands for Relative Error.  Matrix based compression can be determined via predefined relative energy tolerance $\gamma$ (i.e. RE), or set to achieve desired compression, so we performed both.  For the first experiment with $\gamma_1$, the $k_2$ and 
$(k_1,k_2,k_3)$ values that gave the same compression results were 25 and (92,69,69), respectively; for the second experiment these were
 11 and (70, 53,53), respectively.  For completeness: the truncation values for the matrix case corresponding to the approximations chosen as described were 65 and 25 for the $\gamma_1$ experiment and 47 and 11 for the second experiment.  The total number of $\hat{\sigma}_i^{(j)}$ kept in our method for the first experiment was 1729 and 818 in the second.} 

\end{table}

Not only are the actual values in the table relevant to prove the truncated t-SVDMII gives superior quality (in RE) results to both when compression is fixed, but we can actually visualize the impact of the compression across the various approaches.  
In \Cref{fig:car1} and \Cref{fig:car2}, we give the corresponding reconstructed representations of frame 10 and 54 for the 4 methods
under comparable compression ratio
for the second $\gamma$ (columns 2,4-6, second row-block of the table, ie. the results corresponding to the most compression).  You can see that cars disappear altogether and/or artifacts make it appear as though cars may be in the frame when they are not -- that is, at these compression levels, the matrix and tr-HOSVD all suffer from a ghosting effect.   We note that in some frames (not pictured here), the truncated matrix and HOSVD are somewhat sharper than ours, but our approach does not suffer from this ghosting in any frame (it is always clear where the cars are and where they are not).

\begin{figure}[h]
\centering
\subfigure[Original]{\label{fig:1a}
\includegraphics[height=1.60in,width=1.60in]{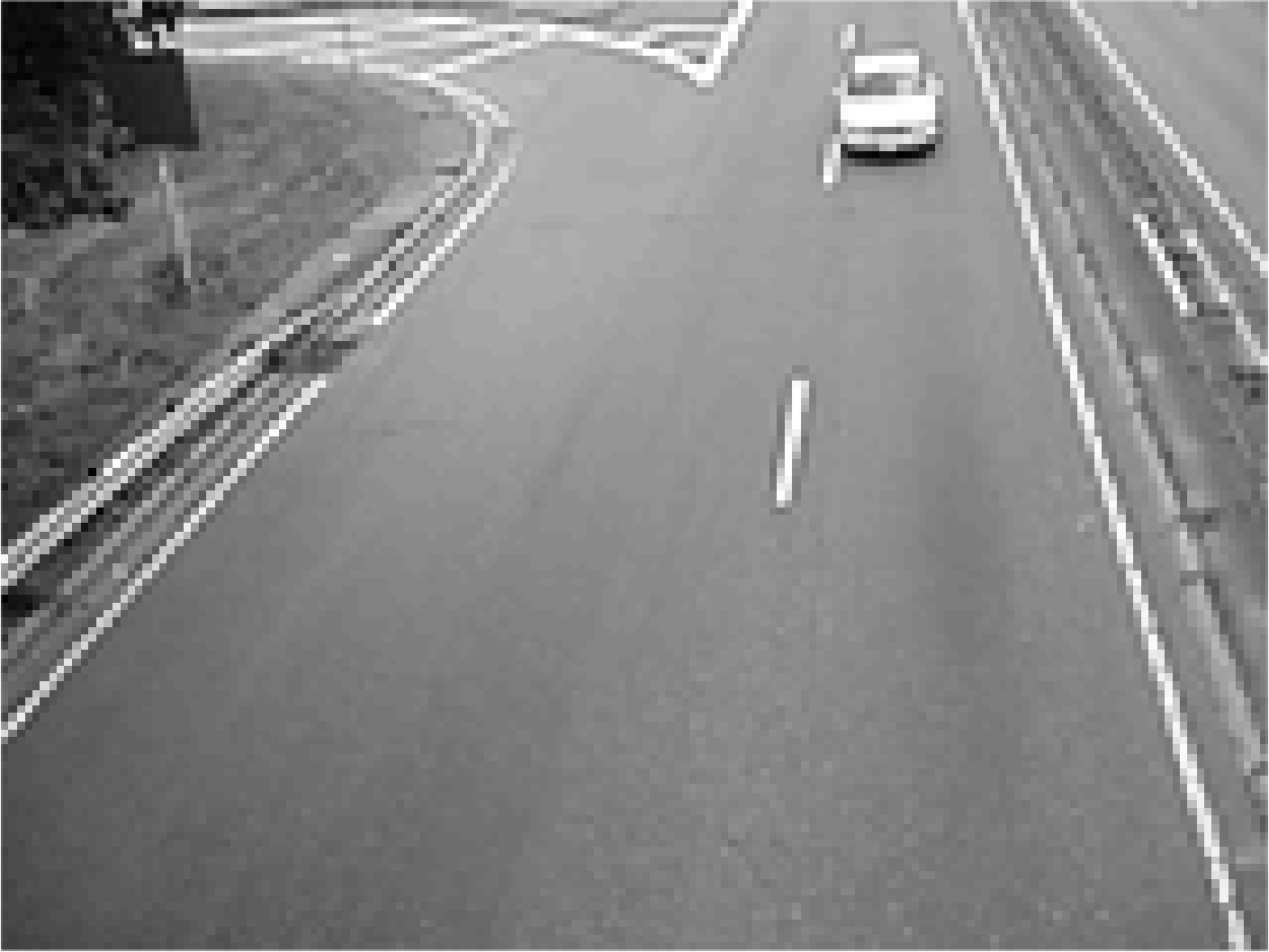}}
\subfigure[tr-tSVDMII]{\label{fig:1b} \includegraphics[height=1.60in,width=1.60in]{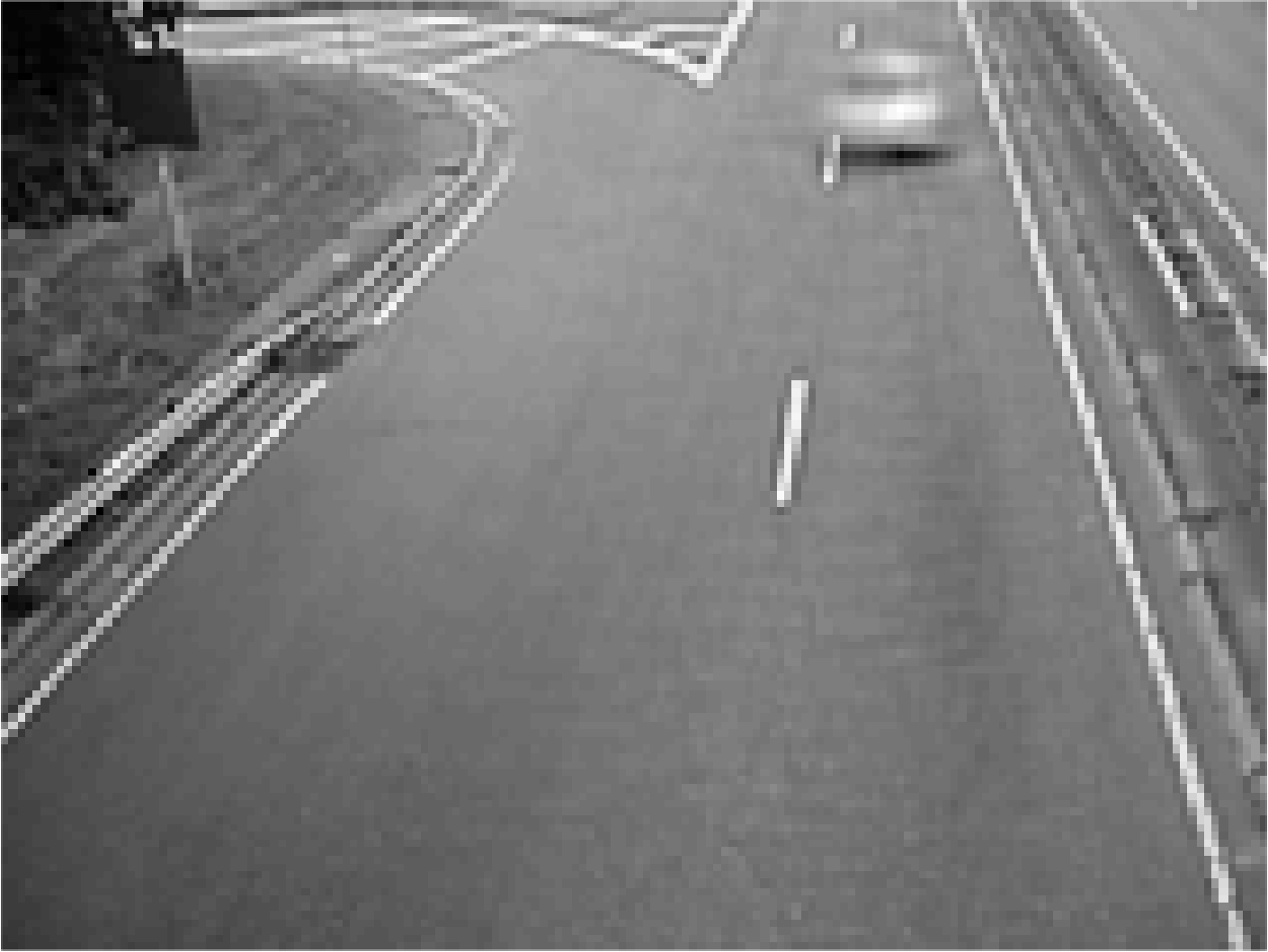}}
\subfigure[tr-Matrix]{\label{fig:1c} \includegraphics[height=1.60in,width=1.60in] {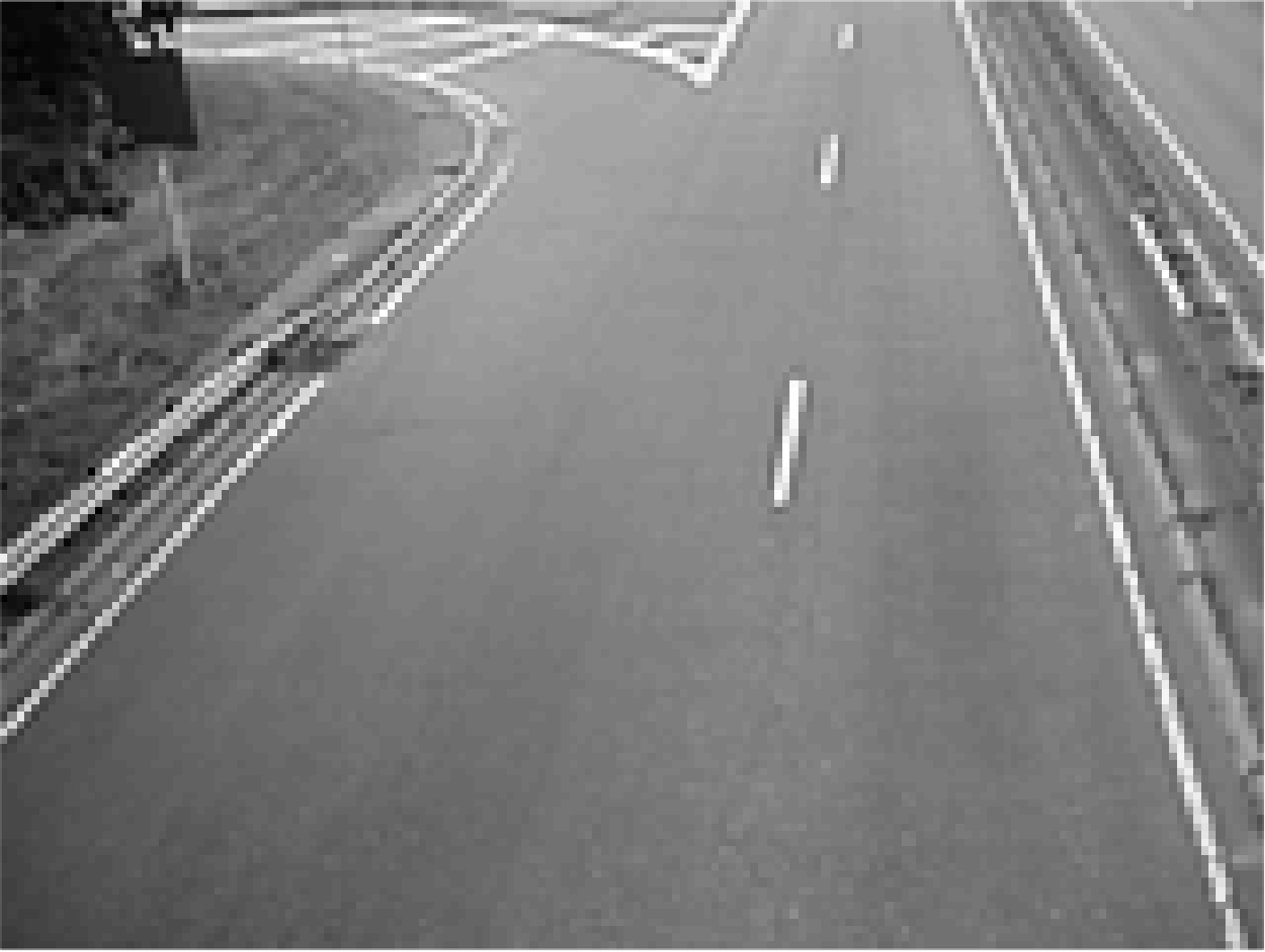}}
\subfigure[tr-HOSVD$(m,25,n)$]{\label{fig:1d} \includegraphics[height=1.60in,width=1.60in]{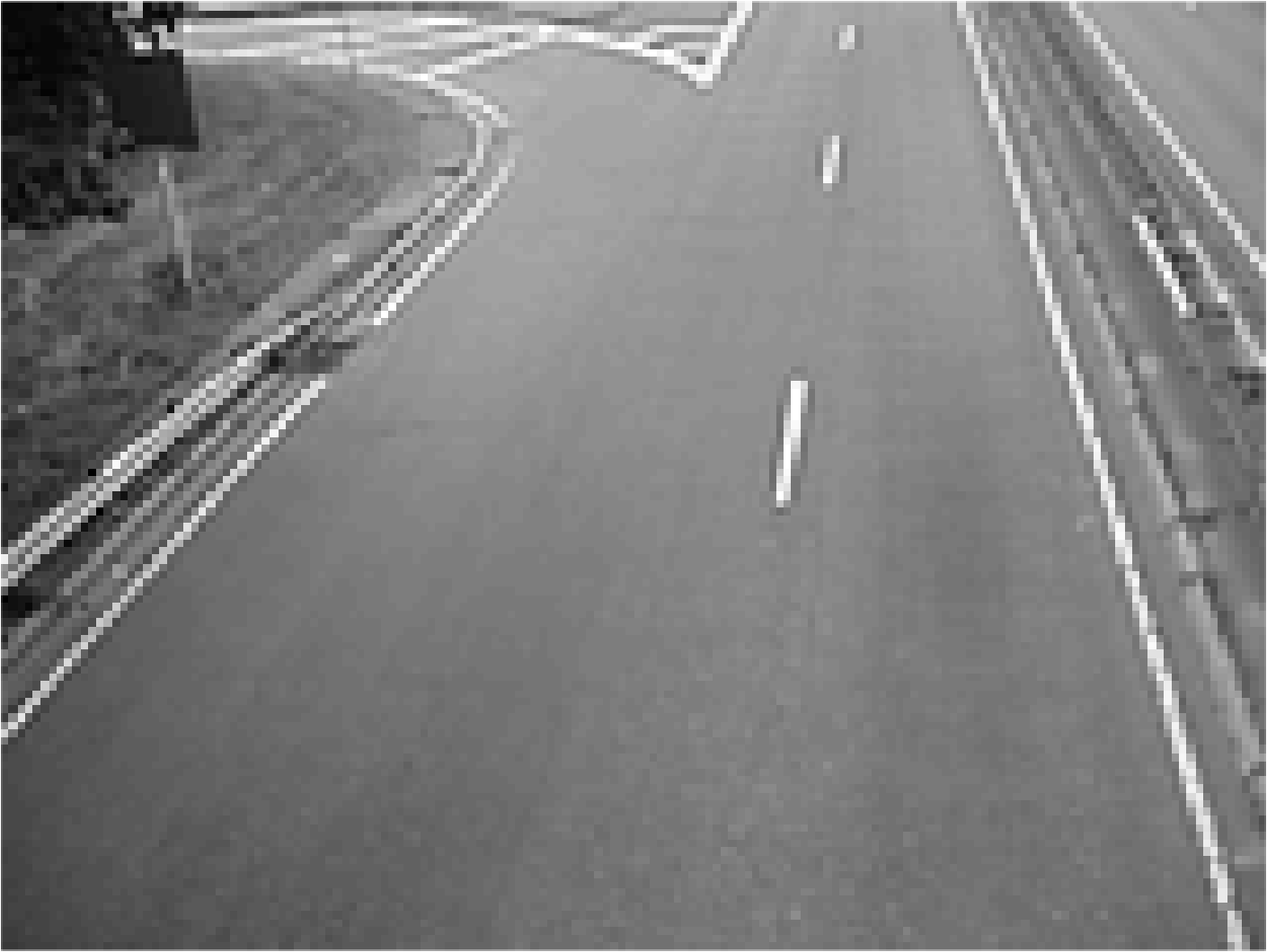}}
\subfigure[tr-HOSVD$(70,53,53)$]{\label{fig:1e} \includegraphics[height=1.60in,width=1.60in]{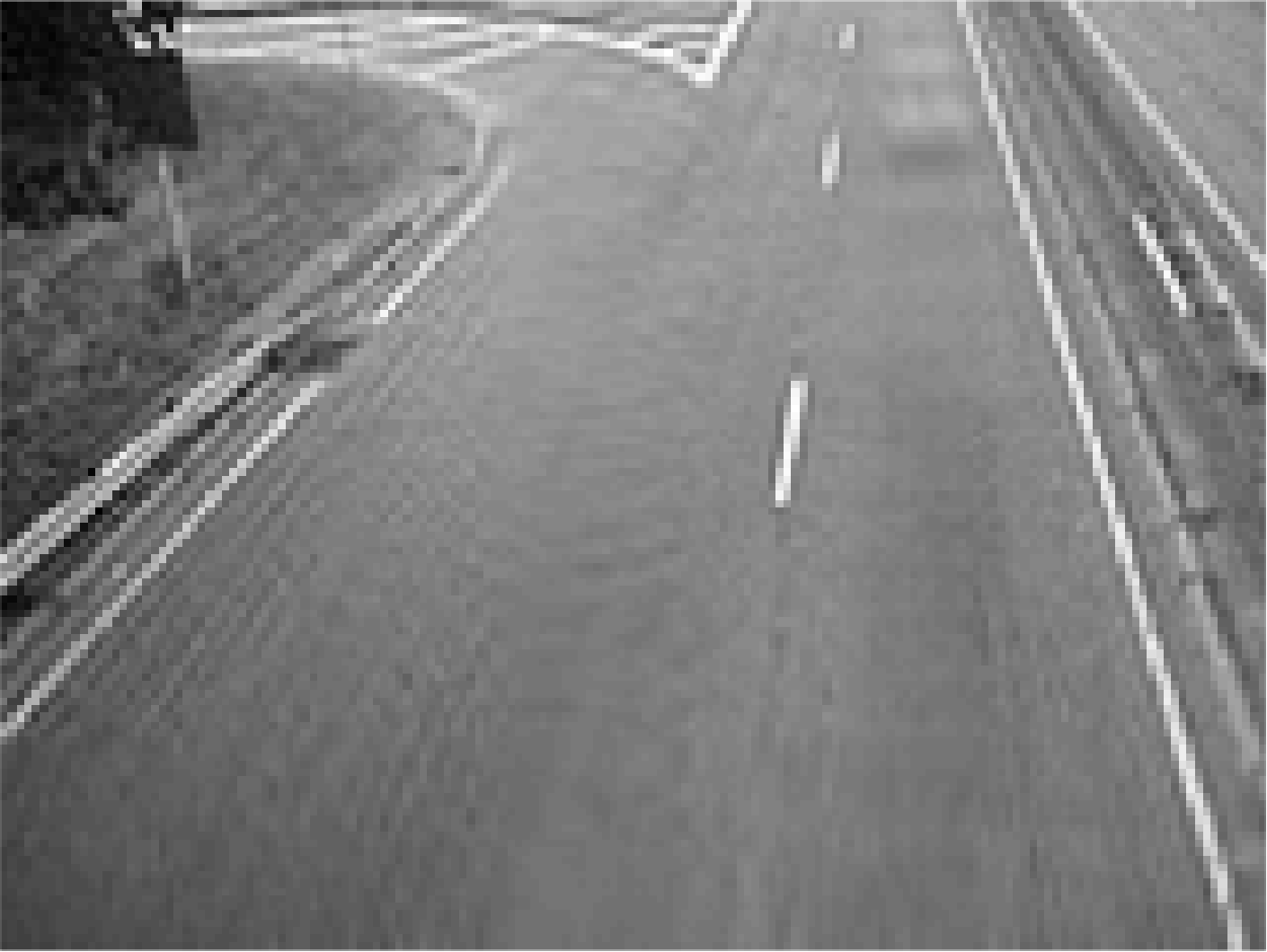}}
\caption{Frame 10, various reconstructions as indicated. }
\label{fig:car1}
\end{figure}

\begin{figure}[h]
\centering
\subfigure[Original]{\label{fig:2a} \includegraphics[height=1.60in,width=1.60in]{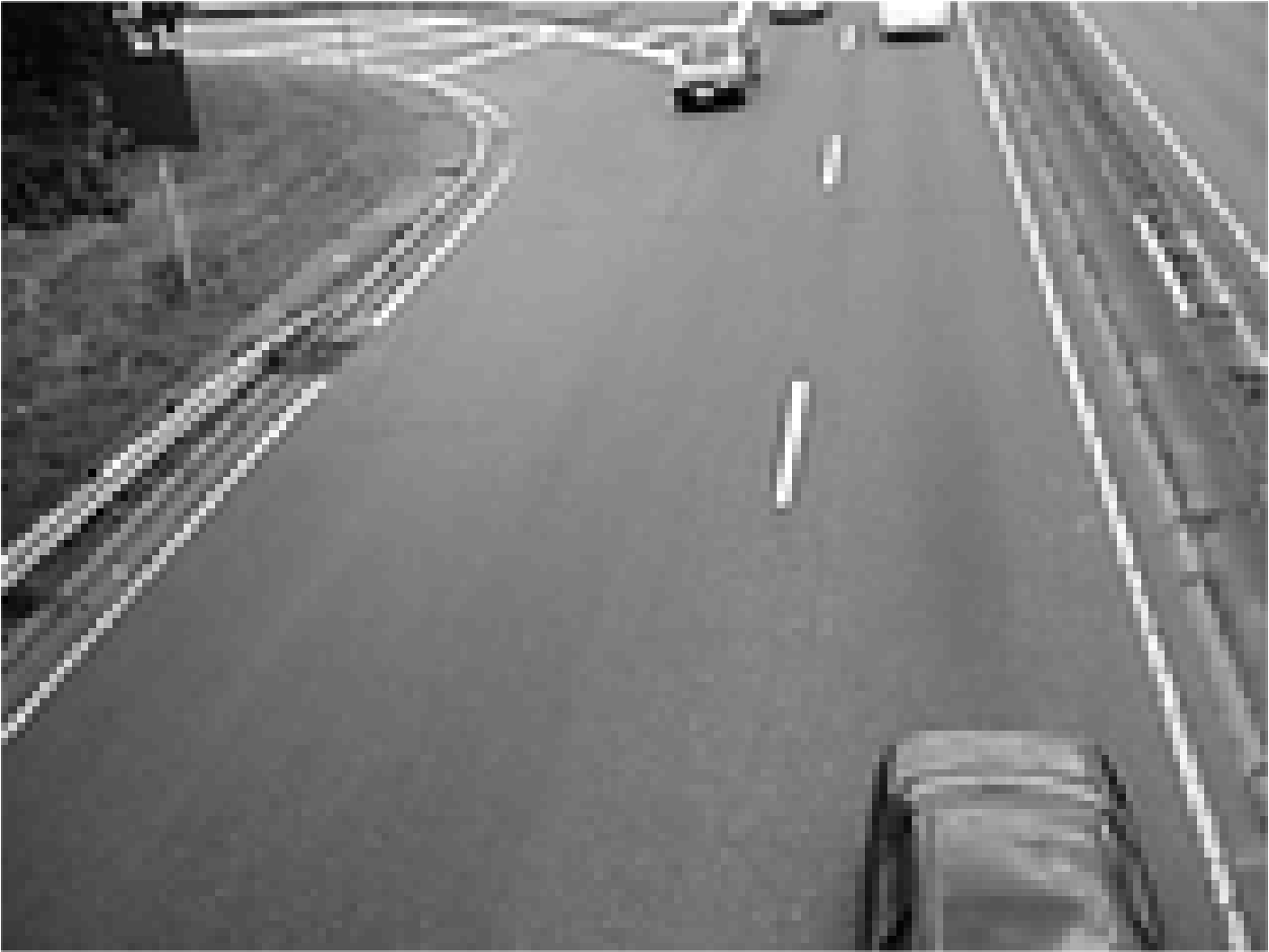}}
\subfigure[tr-tSVDMII]{\label{fig:2b} \includegraphics[height=1.60in,width=1.60in]{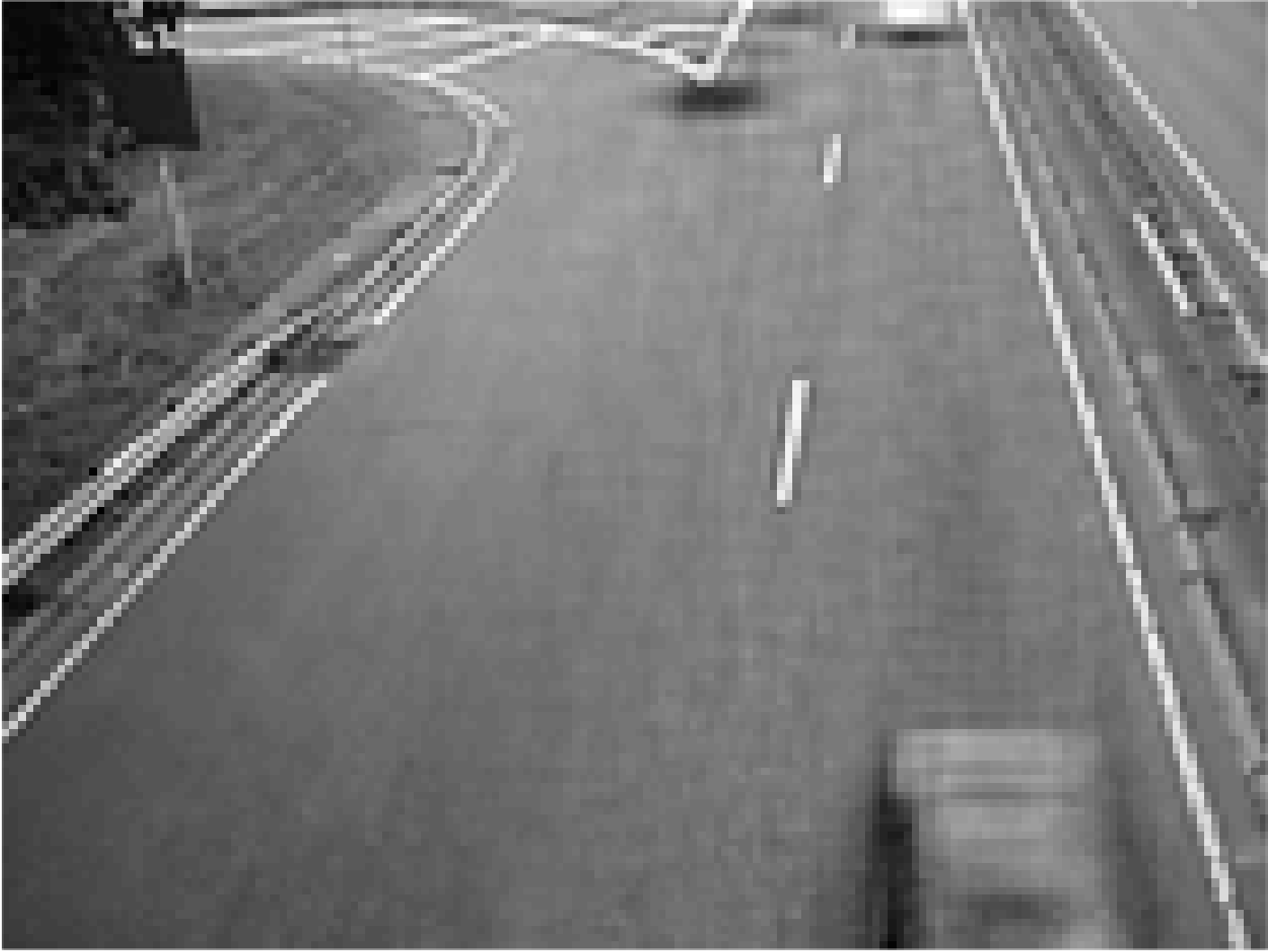}}
\subfigure[tr-Matrix]{\label{fig:2c} \includegraphics[height=1.60in,width=1.60in] {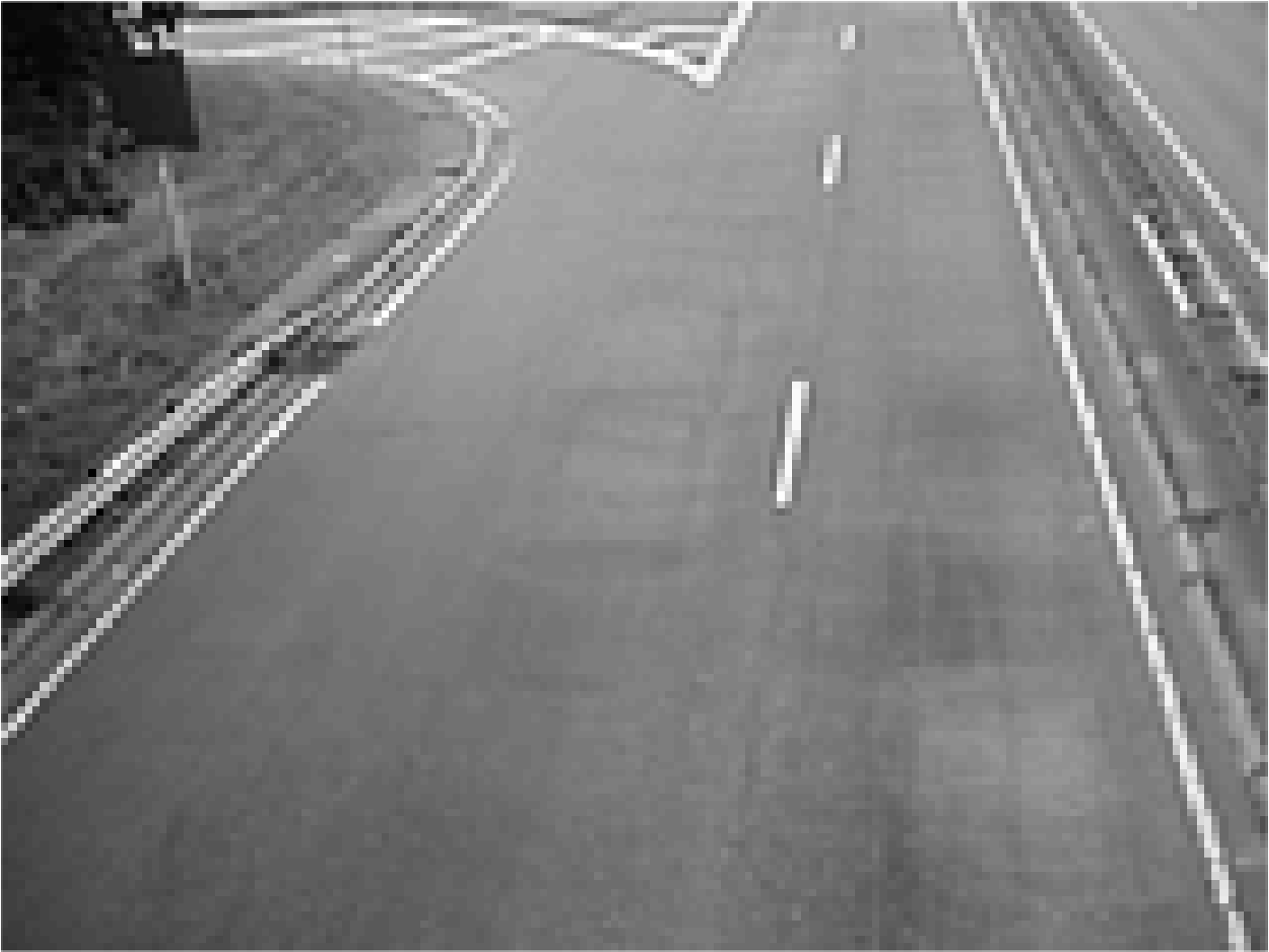}}
\subfigure[tr-HOSVD$(m,25,n)$]{\label{fig:2d} \includegraphics[height=1.60in,width=1.60in]{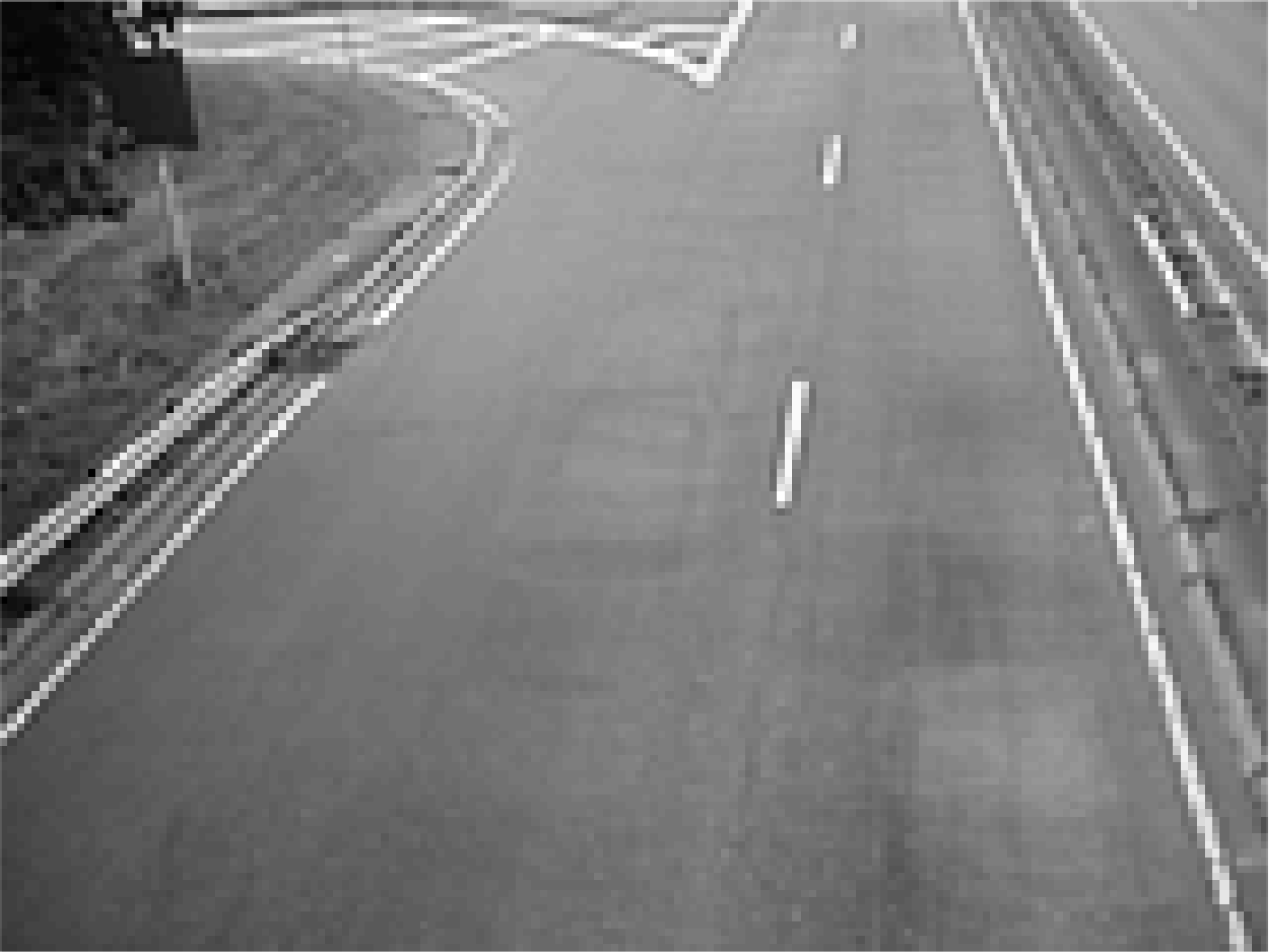}}
\subfigure[tr-HOSVD$(70,53,53)$]{\label{fig:2e} \includegraphics[height=1.60in,width=1.60in]{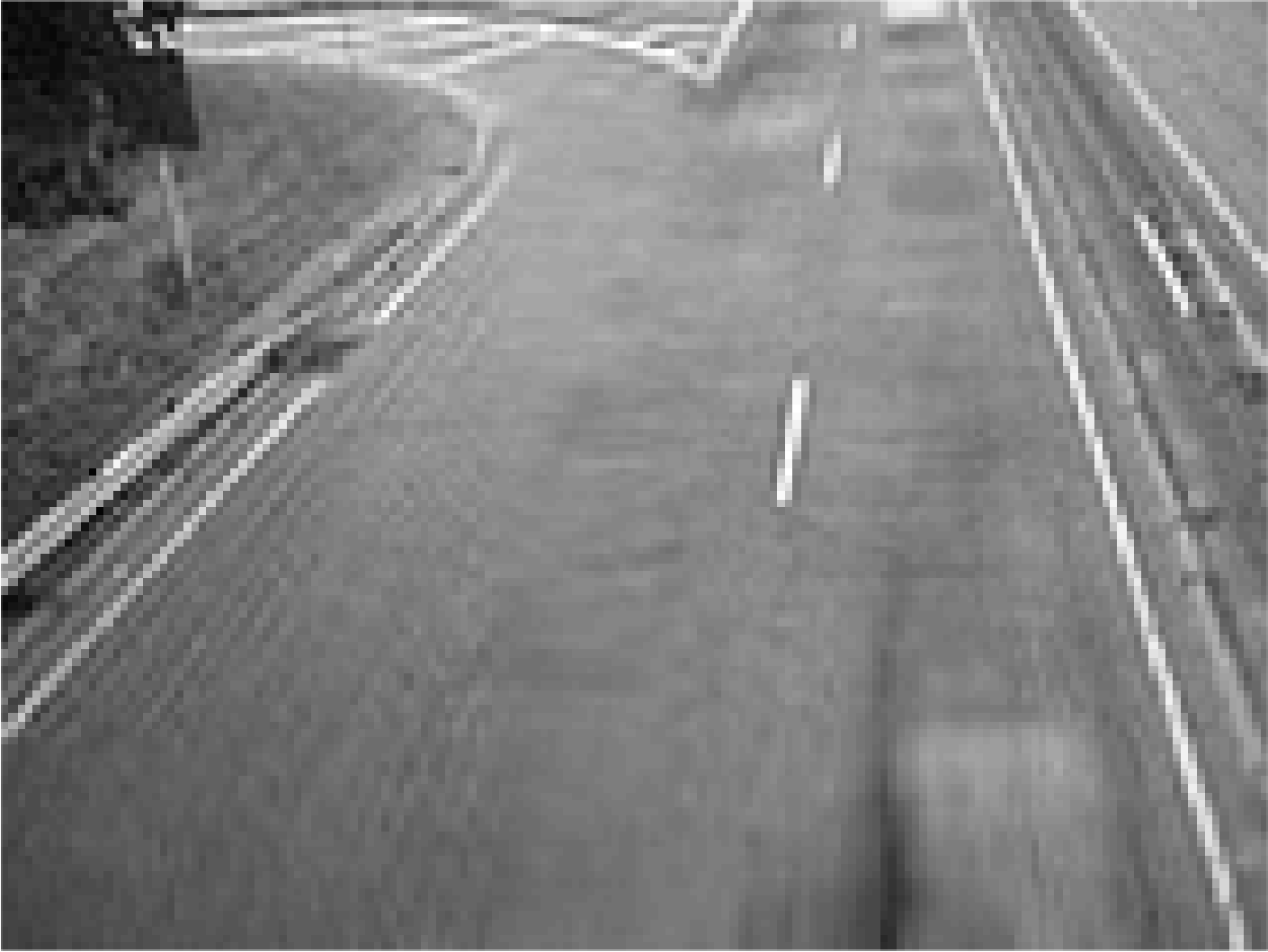}}
\caption{Frame 54, various reconstructions as indicated. }
\label{fig:car2}
\end{figure}


\subsection{Hyperspectral Imagings}

To compare t-SVDMII, truncated matrix SVD, and tr-HOSVD in terms of approximation quality and compressibility, we consider their performance on hyperspectral images.
The hyperspectral dataset consists of $191$ images of size $307 \times 1280$ where each image corresponds to a different wavelength \cite{MultiSpec}.
The images are highly-correlated spatially, and hence are highly compressible.  
We store these images in a $307\times 191\times 1280$ tensor to maximize interaction along the third dimension.

We compare the following compression schemes for the t-SVDMII with $\MM$ being the DFT matrix, the truncated matrix SVD, and the truncated HOSVD in \cref{fig:tensor_compression_comparison}. 
For the matrix SVD, we store the hyperspectral data as a $307\cdot 1280 \times 191$ matrix.  
For the HOSVD, we truncate each dimension proportionally using proportion $p$; that is, $(k_1, k_2,k_3) = (307p,191p,1280p)$.  
	\begin{figure}[H]
	\centering
	\includegraphics[scale=0.28]{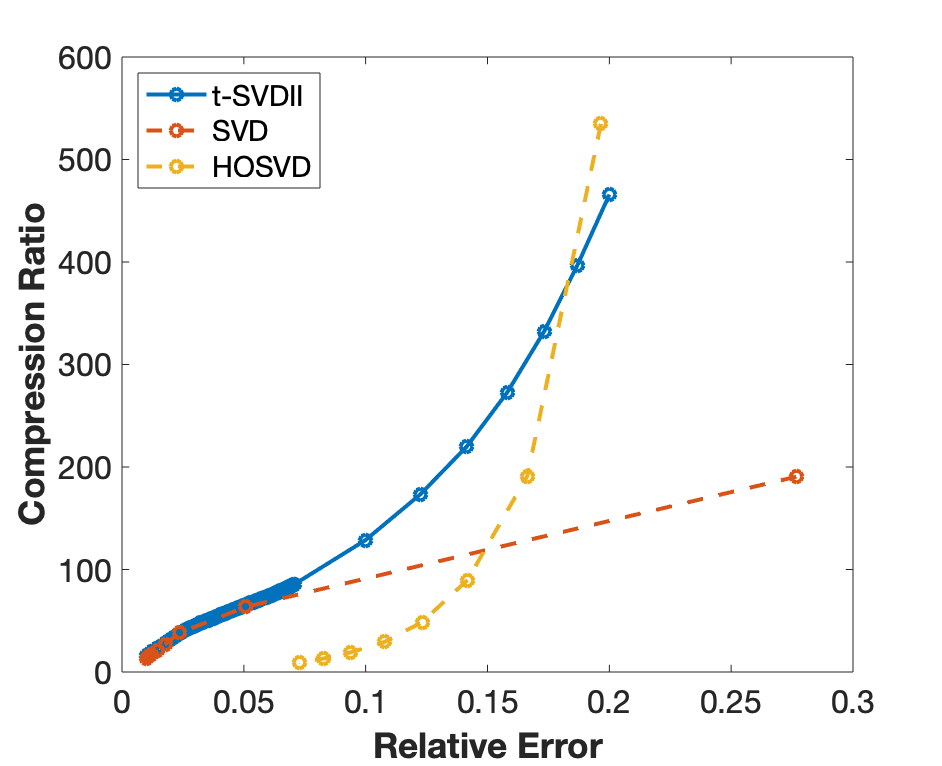}
	\caption{Comparison of various tensor compression techniques for hyperspectral images.  
	For t-SVDMII ($\MM$ is DFT matrix), 
	$\gamma$ ranges between $0.96$ and $0.995$ with step size of $5\times 10^{-3}$ and from $0.995$ to $0.9999$ with step size of $5\times 10^{-5}$.  
	For the matrix SVD, we use $k=1$ to $k$ between $1$ and $15$ with a step size of $2$.
	For the HOSVD, we compress each dimension proportionally with $p$ from $0.125$ to $0.475$ with a step size of $0.05$.  
	We plot the inverse of the compression ratio along the $y$-axis.
	}
	\label{fig:tensor_compression_comparison}
	\end{figure}

The plot in \cref{fig:tensor_compression_comparison} shows the t-SVDMII provides the best representation for the greatest compression. 
The HOSVD becomes competitive with the t-SVDMII for larger relative errors, but to obtain small relative errors, the amount of storage required increases rapidly. Here, it is important to note that the theory 
(e.g. \Cref{cor:comp}) is not directly applicable, since $\MM$ was selected to be the DFT matrix, yet we still see that the t-SVDMII is more highly compressible for the smallest relative errors.    
The truncated matrix SVD is competitive with the t-SVDMII in a relative error sense only at the lowest compression.  This is due to the nature of the hyperspectral data, and the fact that some of the spatial correlations can be identified even in the matricized format.  

In conjunction with \cref{fig:tensor_compression_comparison}, we display the compressed representations of one wavelength of the hyperspectral tensor for the t-SVDMII, the matrix SVD, and the HOSVD for comparable relative errors in \cref{fig:visualization_hyperspectral}.

\begin{figure}[H]
\centering
\small
\begin{tabular}{cccc}
Original & 
	t-SVDII, $\gamma=0.99$
	& Mtx SVD, $k=2$
	& HOSVD, $p=0.38$\\
\includegraphics[scale=0.2]{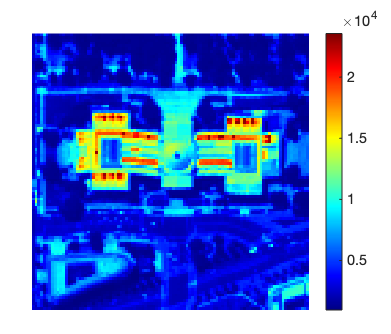} 
&\includegraphics[scale=0.2]{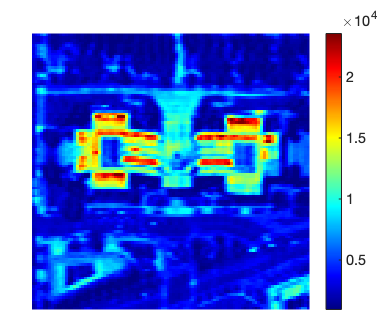} 
	& \includegraphics[scale=0.2]{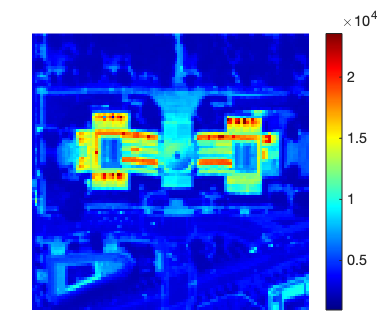} 
	& \includegraphics[scale=0.2]{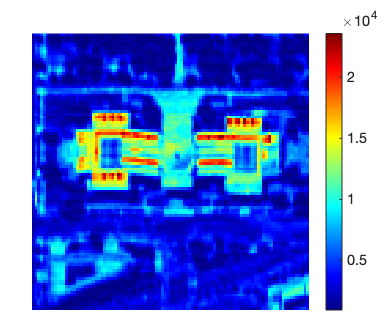} \\
&\includegraphics[scale=0.21]{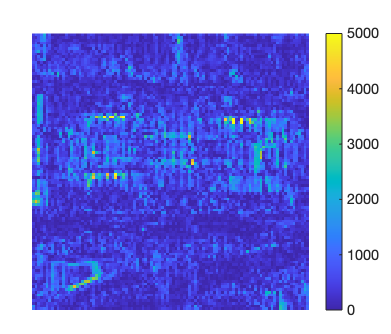} 
	& \includegraphics[scale=0.2]{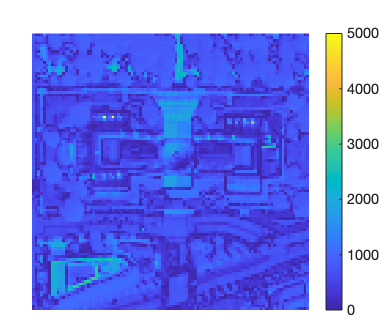}
	& \includegraphics[scale=0.2]{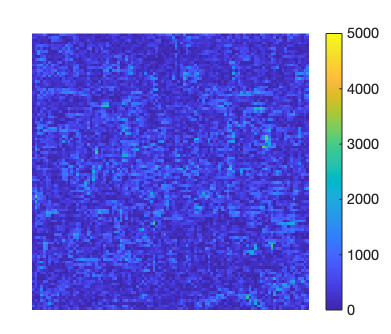}
\end{tabular}
\caption{Comparison of compressed representations with tensors and matrices for hyperspectral images.  The images above are an enlarged section of the hyperspectral image of wavelength $50$.  Top: original and compressed representations.  Bottom: absolute difference between original and compressed representation. 
The $t$-SVDII has a relative error of $0.0999$ and a compression ratio of $128.20$ using the t-product ($\MM$ is is DFT matrix).
The matrix SVD has a relative error of $0.1131$ and a compression ratio of $95.24$.
The HOSVD has a relative error of $0.0938$ and a compression ratio of $19.12$.}
\label{fig:visualization_hyperspectral}
\end{figure}

By examining the absolute difference images in \cref{fig:visualization_hyperspectral}, we notice that the matrix SVD difference image has the most visual similarity to the original image (i.e. important structure is left out of the compressed representation) and the HOSVD has the least.  
The difference image for the t-SVDMII with $\MM$ the DFT matrix has different content altogether, and this accounts for the superior relative error in the matrix case as well as the superior compressibility over the tr-HOSVD.   


\subsection{Extension to 4D and Higher}
                    
Although the algorithms were described for third-order tensors, the algorithmic approach can be extended to higher order tensors since the definitions of the tensor-tensor products extend to higher order tensors in a recursive fashion, as shown in  \cite{Martinetal2013}.  As noted in \cite{LAA}, a similar recursive construct can be used for higher order tensors for the $\starM$ product, or different combinations of transform based products can be used along different modes.  

In \cite{Martinetal2013}, the t-SVD under the t-product is also described for higher order tensors, and ideas based on its truncation for compression and high-order tensor completion and robust PCA can be found in the literature \cite{zhang2014novel,ely2015}.   For the purposes of this study, it suffices to describe the t-SVDMII process from a purely algorithmic point of view.  Note that the process is still parallelizable, and that the size of the matrix factorizations required is only $m \times p$.   Also, as the third and fourth mode matrix products used to transform the tensor into transform space are independent of each other, they can be done in either order, as is convenient.     

\begin{algorithm}[h]
\caption{\label{alg:fourd} 4D t-SVDMII}
\begin{algorithmic}[1]
\STATE  INPUT: $m \times p \times n \times q$ tensor $\TA$; invertible $n \times n$ and $q \times q$ $\MM, \MB$, respectively; energy tolerance $\gamma$. \\
\STATE OUTPUT:  Relevant entries in $\widehat{\TU}, \widehat{\TS}, \widehat{\TV}$ corresponding to $\gamma$ energy.
 \STATE Form $\widehat{\TA}$ by $\widehat{\TA} = \TA \times_3 \MM \times_4 \MB$ \\
 \FORALL{$j=1:q$}
 \FORALL{$i=1:n$}
   \STATE Compute economy matrix factorization $\MW \MD \MQ^*$ of $\widehat{\TA}_{:,:,i,j}$ \\
    \STATE Set $\widehat{\TU}_{:,:,i,j} = \MW$, $\widehat{\TS}_{:,:,i,j} = \MD$, $\widehat{\TV}_{:,:,i,j} = \MQ$; $\Delta(:,i,j)= diag(\MD).^2$ \\
  \ENDFOR 
  \ENDFOR 
 \STATE Sort entries, $\delta_i$ of $\Delta$ in decreasing order \\
 \STATE Compute the partial sums, find $k$ such that $\sum_{i} \delta_i$/$\| \TA \|_F^2 \le \gamma$.  \\
\STATE Save only components in $\widehat{\TU}, \widehat{\TS}, \widehat{\TV}$ that correspond to terms in the $k^{th}$ partial sum.  
\end{algorithmic} 
\end{algorithm} 

We present one set of results here on a subset of YaleB data.   The purpose of this experiment is simply to illustrate proof-of-concept:  we are not claiming that this is the best approach for this data. 

The data consisted of 64, $192 \times 128$ images of 4 people taken at different lighting conditions, for a total of 256 images.  Rather than treat the data as a $192 \times 256 \times 128$ data set, we took each $192 \times 128$ image, and decomposed it into patches of size $x \times y$.   Thus, each sub-image is of size $m = 192/x$ rows and $n = 128/y$ columns, so that the total number of sub-images in each image is $n_{sub}  = mn$.  This data was put into an $m \times 256 \times n \times n_{sub}$ tensor.   We applied both the t-SVDMII and t-SVDM and compare the results.   In each case, note that the size of the matrix SVDs that need to be computed in the first step of the double loop of the algorithm is size $m \times 256$.  
Some results are reported in the table, where we have used $\MM$ and $\MB$ as the unnormalized DFT matrices of size $n$ and $n_{sub}$, respectively.  

\begin{table} \label{tab:four}
\begin{tabular}{|c|c|c|c|c|} \hline
Approach & $k$ & $\gamma$ & Relative Error & Compression Ratio \\  \hline
t-SVDM ($p=6,q=2$)  & 5 &   -- &  0.06  & 5.65 \\  \hline
t-SVDMII ($p=6,q=2$) & -- & .998   &  0.045   & 8.85 \\  \hline
t-SVDM ($p=8,q=2$) & 5 & -- & 0.048 &  4.48  \\  \hline
t-SVDMII ($p=8,q=2$) & -- & .998 & 0.045 & 8.62 \\  \hline
t-SVDM ($p=8,q=2$) & 3 & -- & 0.078 & 7.25 \\  \hline
t-SVDMII ($p=8,q=2$) & -- & .994 & 0.077 & 22.73 \\  \hline
\end{tabular}
\caption{Comparison of t-SVDM and t-SVDMII \Cref{alg:fourd} on 4th order data. 
}
\end{table}
 



           \section{Conclusions and On-going Work}  \label{sec:conclusions}
We have demonstrated theoretically and numerically the significant improvement in compression possible by treating the data in high dimension form and harnessing our tensor-tensor product framework, as opposed to treating the data as a matrix.   
The t-SVDM framework was also particularly useful from a theoretical perspective in interpreting the relationship among HOSVD and tensor-tensor products.  The proofs relied on understanding the latent structure induced under the tensor-tensor products used.  Then we introduced t-SVDMII, which gave additional compression, and we see t-SVDMII outperforms t-SVDM as a compression method in general.  The choice of $\MM$ defining the tensor-tensor product should be tailored to the data for best compression.  Therefore,  consideration as for how to best design $\MM$ to suit the data set, shall be pursued in future  work.

In (\ref{eq:convexII}), we considered the convex combination of t-SVDMII expressions for $\TA$ and $\TA^{\tp}$.  The storage is clearly related to the implicit ranks, $t_1,t_2$, of $\TA_\bfrho$, and $\TA^{\tp}_\bfdel$, respectively:  the first term requires $t_1(n+p)$ and the second $t_2(m + p)$.  In future work, means for optimizing $\alpha$, $k$ and $j$ such that the upper bound on the error is minimized while minimizing the total storage, will be investigated.   Due to the connection between t-SVDMII and CP alluded to in \Cref{ssec:CP}, we postulate  that such an investigation may lead to finding even more compressed CP expressions with improved approximation capability.

\bibliography{references3}{}
\bibliographystyle{siamplain}

\end{document}